\documentclass[10pt]{article}

\usepackage{amsmath,amssymb,amsthm,bm,mathrsfs}
\usepackage{hyperref}
\usepackage{enumerate}
\usepackage[colorinlistoftodos,bordercolor=orange,backgroundcolor=orange!20,linecolor=orange,textsize=normalsize]{todonotes}

\title{Duality attainment and strict feasibility of the generalized moment problem and its relaxations}
\author{Sami Halaseh\footnote{LAAS-CNRS, 7 av du Colonel Roche, 31031 Toulouse Cedex 4, France.} \footnote{Fachbereich Mathematik und Statistik, Universit\"at Konstanz, Germany.} \and   Victor Magron\footnotemark[1] \footnote{Toulouse Mathematics Institute,  118 route de Narbonne, 31062 Toulouse Cedex 9, France.} \and Mateusz Skomra\footnotemark[1]}

\newtheorem{theorem}{Theorem}[section]
\newtheorem*{theorem*}{Theorem}
\newtheorem{lemma}[theorem]{Lemma}
\newtheorem{proposition}[theorem]{Proposition}
\newtheorem{corollary}[theorem]{Corollary}
\newtheorem{assumption}[theorem]{Assumption}
\newtheorem{remark}[theorem]{Remark}
\newtheorem{example}[theorem]{Example}
\numberwithin{equation}{section}

\DeclareMathOperator{\proj}{proj}
\DeclareMathOperator{\cc}{cc}
\DeclareMathOperator{\cs}{cs}
\DeclareMathOperator{\Span}{span}
\DeclareMathOperator{\st}{s.t.}
\DeclareMathOperator{\QM}{QM}
\DeclareMathOperator{\cone}{cone}
\DeclareMathOperator{\sep}{SEP}

\newcommand{\di}{\mathrm{d}}
\newcommand{\Idiv}{\mathrm{I}}
\newcommand{\primal}{\mathrm{P}}
\newcommand{\dual}{\mathrm{D}}
\newcommand{\red}{\mathrm{red}}
\DeclareMathOperator{\relint}{relint}
\DeclareMathOperator{\lin}{lin}
\DeclareMathOperator{\aff}{aff}
\DeclareMathOperator{\Real}{Re}
\DeclareMathOperator{\Imaginary}{Im}
\DeclareMathOperator{\rank}{rank}
\DeclareMathOperator{\Tr}{Trace}
\DeclareMathOperator{\LT}{LT}
\DeclareMathOperator{\lex}{lex}
\DeclareMathOperator{\grlex}{grlex}
\newcommand{\ib}{\mathbf{i}}
\newcommand\R {\mathbb{R}}
\newcommand\g {\mathbf{g}}
\newcommand\p {\mathbf{p}}
\newcommand{\Ip}{I(\mathbf{p})}
\newcommand\C {\mathbb{C}}
\newcommand\D {\mathcal{D}}
\newcommand\N {\mathbb{N}}
\newcommand{\Gb}{\mathcal{G}}
\newcommand\mom {z}
\newcommand{\quantum}{\mathcal{Q}}
\newcommand{\qu}{x}
\newcommand{\qv}{y}

\begin{document}
\maketitle

\begin{abstract}
The generalized moment problem (GMP) is an infinite dimensional linear problem over the cone of finite nonnegative Borel measures.  
When a GMP instance involves finitely many polynomial moment constraints, moment/sum-of-squares hierarchies provide a sequence of bounds converging to the optimal value. 
We consider GMP instances with measures supported over a compact basic semialgebraic set $X$. 
We study 
the case when $X$ has nonempty interior, and 
the case when $X$ is the vanishing set of prescribed polynomials forming a Gröbner basis of the ideal they generate, which we assume is real radical. 
Under a relative interior assumption, we show attainment of the infinite dimensional dual problem, and attainment of each associated finite dimensional sum-of-squares strengthening. For the latter we present two disjoint proofs. The first is obtained by adapting results regarding the closedness of quadratic modules, and the second builds on Csiszár’s work on exponential density constructions to find a strictly feasible measure. 
Finally, we discuss the special case where $X$ is the product of spheres, and applications of our results to GMP instances arising from tensor optimization and quantum information theory.


\end{abstract}

\section{Introduction, notation, and problem statement}

The generalized moment problem (GMP) is an infinite dimensional linear program (LP) over the cone of finite nonnegative Borel measures supported on a set $X \subset \R^n$. 
Given $m \in \N$, $b \in \R^m$, and polynomials $f, h_1,\dots,h_m\in\R[x_1,\dots,x_n] =\R[x]$, we consider in this paper the following GMP instance:
\begin{equation}\label{eq:gmp_primal}\tag{GMP-P}
\inf_{\mu} 
\left\{
    \int f \, d\mu \,
    \colon
    \int h_i \, d\mu = b_i, \; i=1,\dots,m, 
    \; \mu \in \mathcal{M}_+(X)
\right\},
\end{equation}
where  $\mathcal{M}_+(X)$ is the set of finite nonnegative Borel measures supported on $X$. To ensure that each feasible solution of \eqref{eq:gmp_primal} has bounded mass, we will always assume in the sequel that $h_1 \equiv1$ and $b_1>0$. 

This framework encompasses classical moment problems, rational function optimization, as well as various applications arising from tensor optimization and quantum information theory. An important special case is polynomial optimization, where a polynomial $f$ is minimized over a  semialgebraic set $X$. This corresponds to a GMP instance with a single constraint $(m = 1)$ given by $h_1\equiv 1$ and $b_1=1$. 
The moment/sum-of-squares (moment/SOS) hierarchy of semidefinite relaxations introduced by Lasserre \cite{lasserre2001global} provides a sequence of bounds converging from below to the minimum of $f$ on $X$. 
Each relaxation is a semidefinite program \cite{vandenberghe1996semidefinite}, i.e., it involves a linear objective function and linear matrix inequality constraints. 
The task of computing the optimal value of such semidefinite programs can be done efficiently via primal-dual interior-point methods. 
These semidefinite relaxations have become a cornerstone of modern polynomial optimization. This approach can be extended to a general GMP with polynomial constraints, see \cite{lasserre2008semidefinite}. 
General convergence results of the moment/SOS hierarchy for polynomial GMP instances under broad assumptions were established in \cite{tacchi2022convergence}.
For more details on applications of the GMP, especially with polynomial input data, see the monograph \cite{lasserre_momentsandapps_2010}. 

The dual of \eqref{eq:gmp_primal} is the following infinite dimensional LP over nonnegative polynomials:
\begin{equation}\label{eq:gmp_dual}\tag{GMP-D}
\sup_{\lambda} \left\{ \sum_{i=1}^m b_i \lambda_i \colon f - \sum_{i=1}^m h_i\lambda_i \geq 0 \text{ on } X,\; \lambda \in \R^m \right\}.
\end{equation}
We say that \eqref{eq:gmp_primal} has \emph{duality attainment} when the supremum in \eqref{eq:gmp_dual} is actually a maximum. Given $h_1, \dots, h_m$ as above, we define the convex cone as
\begin{equation}\label{eq:K_cone}
K := \left\{ \left(\int_X h_i \di \mu\right)_{i=1,\dots,m} \mid  \mu \in \mathcal{M}_+(X) \right\}.
\end{equation}
A crucial assumption for our result is the relative interiority of the $b$ vector, i.e., that $b\in \relint(K)$. 

\subsection{Contributions}\label{sec:contributions}

Our first main result shows that duality attainment holds for the infinite dimensional GMP.
\begin{theorem}
\label{th:gmp_int}
Assume that $h_1\equiv1$, $b_1 > 0$ and that $b = (b_1,\dots,b_m)$ is in the relative interior of $K$. 
Then the dual program \eqref{eq:gmp_dual} attains its optimal value. 
\end{theorem}

The above generalizes the duality attainment result of Nie in \cite[Proposition 3.6]{nie_atruncKprob_2014}, which we discuss further in Remark \ref{rk:gmp_int}. Theorem \ref{th:gmp_int} also complements the result by de Klerk and Laurent \cite[Theorem 1]{klerk2019survey}, which we state in our setting in Theorem \ref{th:nogap_gmp}.

Our second main result shows that duality attainment holds for the finite dimensional semidefinite relaxations of \eqref{eq:gmp_primal}, which we now introduce. For further details, see \cite{laurent_survey_2009,lasserre_momentsandapps_2010,nie2023moment} for thorough expositions.

When discussing relaxations, we always make the following assumption:
\begin{assumption}
\label{hyp:ball}
Let $p_1,\dots,p_s,g_1\dots,g_{\ell}$ be  polynomials in $\R[x]$. The set $X$ is a basic compact semialgebraic set with one of the two following descriptions:
\begin{enumerate}[(i)]
\item $X = \big\{x\in \R^n \mid p_1(x),\dots,p_s(x) =0 \;\; \text{and}\;\; g_1(x)\geq 0 \big\}$
\item $X = \big\{x\in \R^n \mid g_1(x),\dots,g_\ell(x) \geq 0\big\},$
\end{enumerate}
where $g_1= R - \|x\|^2$ for some $R>0$, such that, in (i), $V_{\R}(\Ip) := \{x\in \R^n \mid p_1(x),\dots,p_s(x) =0\} \subset \big\{x\in \R^n \mid  g_1(x) = R - \|x\|^2 > 0\big\}$.
\end{assumption}

We introduce the notions needed to formulate the relaxation of \eqref{eq:gmp_primal}. Given $t \in \N$ and a multisequence $\mom \in \R^{\N^n_{2t}}$, we define the \emph{Riesz functional} of $\mom$ as
\[\begin{aligned}
L_{\mom}: \R[x]_{2t}  &\longrightarrow \;\R \\
 \sum_{\alpha} p_{\alpha} x^{\alpha} \; &\longmapsto \; \sum_{\alpha} p_{\alpha} \mom_{\alpha}.
\end{aligned}\]
The \emph{truncated pseudo-moment matrix} $M_t(\mom)$ associated to $\mom$ is defined by 
\[
M_t(\mom) := \big(L_{\mom}(x^{\alpha+\beta})\big)_{\alpha,\beta \in \N^n_t} = (\mom_{\alpha+\beta})_{\alpha,\beta \in \N^n_t}.
\]
Furthermore, given a polynomial $k = \sum_{\gamma} k_{\gamma} x^{\gamma} \in \R[x]_{2t}$, set $t_k = \lceil\frac{\deg k}{2} \rceil$. We define the \emph{shifted sequence}
\begin{equation}\label{eq:shifted_seq}
(k\mom)_{\alpha} = L_\mom(k x^{\alpha}) = \sum_{|\gamma|\leq \deg(k)} k_{\gamma} L_{\mom}(x^{\alpha+\gamma}) = \sum_{\gamma} k_{\gamma} \mom_{\alpha+\gamma}, \;\; \text{with}\; \alpha\in \N^n_{2(t-t_k)}.
\end{equation}
This in turn allows us to define the \emph{truncated localizing matrix} associated to $k$ and $\mom$ as the pseudo-moment matrix of the sequence $kz$,
and is written as $M_{(t-t_k)}(k \mom)$. Thus, 
\[
M_{(t-t_k)}(k \mom) = \big(L_\mom(k x^{\alpha+\beta})\big)_{\alpha,\beta\in \N_{(t-t_k)}^n}.
\]

The following notions are needed for the formulation of the strengthenings of \eqref{eq:gmp_dual}. Set $\g=\{g_1,\dots,g_{\ell}\}$, and $t_{g_j} = \lceil\frac{\deg g_j}{2} \rceil$. 
With $t \geq \max \{ t_{g_1},\dots,t_{g_\ell} \}$, the \emph{truncated quadratic module} $\QM_{2t}(\g)$ generated by $\g$ is defined by
\[
\QM_{2t}(\g) = \left\{ \sigma_0 + \sum_{j=1}^{\ell} \sigma_j g_j \mid \sigma_0 \in \Sigma[x]_{2t}, \; \sigma_j \in \Sigma[x]_{2(t-t_{g_j})}, \;\; j=1,\dots,\ell, \right\},
\]
where $\Sigma[x]_{2t}$ is the set of sum-of-squares (SOS) polynomials of degree at most $2t$. The quadratic module generated by $\g$ is $\QM(\g) = \bigcup_{t \in \N} \QM_{2t}(\g)$. Similarly, set $\p = \{p_1,\dots,p_s\}$, and $t_{p_l} = \lceil\frac{\deg p_l}{2} \rceil$. The \emph{ideal} generated by $\p$ is denoted by $\Ip$, and its truncation is defined as
\begin{equation}\label{eq:trunc_ideal}
\langle \p \rangle_{2t} := \left\{\sum_{l=1}^s k_l p_l \mid \deg(k_l p_l) \leq 2 t\right\}.
\end{equation} 

With $f, h_1,\dots,h_m \in \R[x]$ and the above notation, define 
\begin{equation}
\label{eq:tmin}
t_{\min} := \max \left\lbrace t_{g_1},\dots,t_{g_\ell}, t_{p_1},\dots,t_{p_s}, \Big\lceil\frac{\deg h_1}{2} \Big\rceil, \dots , \Big\lceil\frac{\deg h_m}{2} \Big\rceil,  \Big\lceil\frac{\deg f}{2} \Big\rceil \right \rbrace.
\end{equation}
We now formulate the hierarchy of relaxations of \eqref{eq:gmp_primal}, indexed by $t \geq t_{\min}$,
\begin{equation}\label{eq:gmp_relax}\tag{GMP-P-rel}
\begin{aligned}
\primal_t = \; &\inf_{\mom}  \; \; L_{\mom}(f)\\
	& \st \;  L_{\mom}(h_i) = b_i, \;\; i=1,\dots,m \\
	& \qquad M_t(\mom)\succeq 0,\; M_{t-t_{g_j}}(g_j\mom) \succeq 0,\; M_{t-t_{p_l}}(p_l \mom) = 0, \\
	& \qquad j = 1,\dots,\ell,\;\;l=1,\dots,s,\;\;\mom \in \R^{\N^n_{2t}},
\end{aligned}
\end{equation}
where `$\succeq0$' denotes positive semidefiniteness, and, later, `$\succ 0$' denotes positive definiteness.
Since the cost function and constraints of \eqref{eq:gmp_relax} depend linearly on $\mom$, \eqref{eq:gmp_relax} is a \emph{semidefinite program}  \cite{vandenberghe1996semidefinite}. 
A solution $\mom$ is \emph{strictly feasible} if $M_t(\mom)\succ 0$ and  $M_{t-t_{g_j}}(g_j\mom) \succ 0$ ($j=1,\dots,\ell$).
The dual hierarchy of \eqref{eq:gmp_relax} is obtained by strengthening the nonnegativity constraints of \eqref{eq:gmp_dual}, and is given by the following sequence of semidefinite programs indexed by $t \geq t_{\min}$,
\begin{equation}\label{eq:dual_relax}\tag{GMP-D-str}
\begin{aligned}
\dual_t = &\sup_{\lambda \in \R^m}  \; \; \sum_{i=1}^m b_i \lambda_i \\
	& \st \; f - \sum_{i=1}^m  h_i \lambda_i \in  \QM(\g)_{2t}+\langle \p \rangle_{2t}.
\end{aligned}
\end{equation}
The problems above form a pair of primal-dual semidefinite programs. 
We now state our second main result, regarding the duality attainment of the finite dimensional relaxations. 
\begin{theorem}\label{th:attainSOS}
Let $X$ be a compact subset of $\R^n$ satisfying Assumption \ref{hyp:ball}. 
Suppose that either
\begin{enumerate}[(i)]
    \item The ideal $\Ip=(p_1,\dots,p_s)$ is real radical, and that $\p=\{p_1,\dots,p_s\}$ forms a Gröbner basis of $\Ip$ with respect to the graded lexicographic order, or
    \item $\p=\{0\}$ and $X$ has nonempty interior. 
\end{enumerate}
Assume further, that $h_1\equiv1$, $b_1 > 0$, and that $b = (b_1,\dots,b_m)$ is in the relative interior of $K$. 
Then the semidefinite program \eqref{eq:dual_relax} attains its optimal value for all $t \geq t_{\min}$.
\end{theorem}

See Section \ref{sec:examples} for the definitions of a real radical ideal and a Gröbner basis. 
We provide two disjoint proofs of the above theorem, one by studying the dual program directly, and one by showing the existence of a Slater point for the primal program, which in this case, is a strictly feasible solution. We give context on how these proof techniques fit into the existing literature in Subsection \ref{ssec:related_works}. 

Theorem \ref{th:attainSOS} complements the following result by Tacchi, which we reformulate in our setting as follows.
\begin{proposition}[{\cite[Proposition 6]{tacchi2022convergence}}]
\label{prop:nogaprelax}
Suppose that $h_1 \equiv 1$, $b_1 > 0$ and that Assumption \ref{hyp:ball} holds. 
If \eqref{eq:gmp_relax} has a feasible solution, then there is no duality gap between \eqref{eq:gmp_relax} and \eqref{eq:dual_relax}. 
\end{proposition}

\paragraph{Significance of duality attainment and strict feasibility}

The theorems above address the important matter of duality attainment for the GMP, i.e., under what conditions does the GMP dual problem admit an optimal solution? 
While weak duality holds automatically, the existence of optimal dual solutions is delicate and depends on geometric properties of the underlying generalized moment cone. 
Dual attainment for polynomial GMP instances guarantees the existence of optimal polynomial certificates in the GMP dual.
On the practical side, one focuses on the moment/SOS hierarchy of semidefinite relaxations. 
Dual attainment at a given relaxation order implies that one can obtain an optimal sum-of-squares certificate with bounded coefficients. 
In most practical applications, these certificates correspond to explicit algebraic witnesses. 
Ensuring their existence is therefore essential both theoretically and computationally. 
A second matter is about strong feasibility. 
Modern semidefinite solvers typically implement primal-dual interior-point algorithms. 
These solvers tend to exhibit their most stable and predictable convergence behavior when the original program has a pair of primal-dual strictly feasible solutions. 
All together, strong feasibility and strong duality ensure that the iterates remain well-centered and that the duality gap decreases reliably, which is key to the practical efficiency of interior-point methods; see \cite[Section 7.1]{helmberg2002sdp} and \cite[Section 5]{vandenberghe1996semidefinite} for more details. 

\subsection{Related works and methodology}\label{ssec:related_works}

In Theorem \ref{th:gmp_int}, we show that duality attainment holds for the infinite dimensional GMP. This generalizes the result in \cite{nie_atruncKprob_2014}, where optimality conditions and dual attainment for the $A$-truncated moment problem are studied. There, Nie shows, under a similar relative interior assumption, that duality attainment is obtained for the infinite dimensional GMP. See Remark \ref{rk:gmp_int} for further discussion.

Theorem \ref{th:attainSOS} generalizes the duality attainment result \cite[Proposition 3.10]{baldi2025exact}, which is in the polynomial optimization context. The techniques we use to provide two proofs of Theorem \ref{th:attainSOS} fit into respective lines of research, and so, in the following, we discuss the two proof techniques, and their related works. We then discuss an algebraic reduction method common to both proofs, and the related works thereof.

\paragraph{Closedness of quadratic modules}
Our first proof of duality attainment of the relaxations comes from directly studying the dual semidefinite strengthenings. In that context, the feasible set is the quadratic module generated by the polynomials defining the support set. Hence, showing that the quadratic module is closed can be used to show duality attainment. Our result on duality attainment via the closedness of the quadratic module generalizes results from a line of work by Powers and Scheiderer in \cite[Proposition 2.6]{powers2001noncompact}, Marshall in \cite[Theorem 3.1]{marshall_optim_2003}, Schweighofer in \cite[Corollary 20]{schweighofer2005compact}, and Laurent in \cite[Section 6.2]{laurent_survey_2009}. Such an approach was used in the context of duality attainment and duality gaps in \cite[Section 2]{sekiguchi2013duality}. These results were also used to show duality attainment type results in special empty interior cases such as the vanishing sets of the gradient ideal in \cite[Theorem 8]{nie2006gradideal}, the KKT ideal \cite[Theorems 3.1 and 3.2]{demmel2007kktideal}, and more general ideals in \cite[Proposition 2.5]{nie2013varieties}. Our result extends these cases.

\paragraph{Truncated moment problems and entropy methods}
Our second proof of duality attainment of the relaxations comes from showing the existence of a strictly feasible point for the primal relaxation. This employs the strong duality of semidefinite programs.

Like our initial GMP instance, maximum entropy problems are often considered as \emph{partially-finite programs}, i.e., convex programs with a single infinite dimensional variable subject to a finite number of linear constraints. In fact, duality and existence results are available for such partially-finite programs. In particular, the entropy-based estimation problem has long been studied to recover a measure density from finitely many moment constraints. We mention some related works that highlight the deep connection between entropy maximization, interiority conditions, and duality attainment in measure theoretic optimization.

Attainment of the dual for such partially-finite programs is proved in \cite[Corollary 2.6]{borwein1991duality}, and the existence of a (unique) optimal density for the primal is derived in \cite[Theorem 4.8]{borwein1991duality}, the authors use an equivalent assumption to the relative interior assumption that we consider; see \cite[Lemma 3.3]{borwein1993partially}. Establishing these duality and existence results requires elaborate ingredients from (infinite dimensional) convex analysis, in particular the study of relevant Legendre--Fenchel transform domains and their subgradients. 

Works leading to comparable duality and density existence results rely on tools from probability theory, in particular the so-called \emph{maximal entropy method on the mean} (MEM), which is developed in \cite{dacunha1990maximum,gamboa1997bayesian} and is based on sequences of Bayesian problems. 
The MEM is a construction where the moment problem is approximated by a sequence of finite dimensional problems, which are obtained by a discretization of the underlying support space $X$. The MEM estimator is the limit of these discretized estimators. 

The case where $X$ is compact is also studied in \cite{blekherman2012truncated}. There, it is shown that for a measure $\mu$ on $X$, absolutely continuous with respect to the Lebesgue measure on $X$, any moment sequence in the interior of the moment cone has a representing measure $\nu$, which is absolutely continuous with respect to $\mu$. Moreover, the associated Radon--Nikodym derivative has an exponential form. Similarly to \cite{borwein1991duality}, the main proof ingredient used in \cite{blekherman2012truncated} is that the relative interior of the truncated moment cone is actually the domain of the Legendre--Fenchel transform associated with a certain convex function. 
Computing the value of the Legendre--Fenchel transform boils down to solving a maximum (Boltzmann--Shannon) entropy  problem.  

It turns out that this previously established result about the existence of a feasible exponential density for the GMP can be alternatively proved in a more concise fashion, under the relative interior assumption made in Theorem \ref{th:density_gmp}. The resulting density that we obtain is the same as in the above mentioned works \cite{borwein1991duality,gamboa1997bayesian,blekherman2012truncated}. Rather than relying on elaborate Legendre--Fenchel duality or Bayesian sequences, our approach solely leverages entropy based constructions and $\Idiv$-projection techniques inspired by Csiszár’s framework \cite{csiszar1975divergence,csiszar2001convex}. The existence of an exponential density feasible for the primal GMP is stated in Theorem \ref{th:density_gmp}, and can be seen as an extension of the result in \cite{blekherman2012truncated} beyond the monomial setting.  

We believe that this alternative self-contained proof is of independent value for modern researchers interested in theoretical and practical aspects of generalized moment problems, due to the fact that it proves the existence of a strictly feasible solution.

\paragraph{Algebraic reduction}
In both of the techniques above, dealing with the case of equality constraints, where the support set $X$ will have empty interior (see Assumption \ref{hyp:ball}(i)), requires the use of an algebraic reduction of the moment relaxation and its dual SOS strengthening. On the dual side, such a reduction amounts to finding SOS certificates in the quotient ring of the ideal. This places our work in a line of research regarding the exploitation of equality constraints to reduce the size of the finite dimensional programs. The technique of using an algebraic reduction was also employed in \cite{henrion2012innerapprox}.

In \cite{lasserre2002binary}, binary programming is modeled by using $x_i^2=x_i$ equality constraints. In \cite{parrilo2002finitesets}, finite varieties and certificates using their ideals are considered. Furthermore, a real radical assumption is made, and the need for a Gröbner basis and an associated quotient basis are briefly discussed. Later, in \cite{marshall_optim_2003}, a high level approach is presented for the algebraic reduction by an arbitrary ideal. This is contrasted by the detailed treatment of the algebraic reduction in the case of $X$ being zero dimensional and finite in \cite{laurent2007varieties}. In \cite{pena2008ideal}, the truncation of the ideal is considered as part of the set giving a positivity certificate, here the variety does not need to be finite. In \cite{vo2008equalities}, a heuristic is developed on how to use equality constraints to contract the semidefinite programs associated to the moment relaxations.

Reductions and certificates over special ideals and bases were also studied in the above mentioned works \cite{nie2006gradideal,demmel2007kktideal,nie2013varieties}, and also in \cite{bucero2016border}, which employs a reduction using the border basis. The reduction we make is in the style of the one in \cite[Section 2]{laurent2007varieties}, and generalizes the above mentioned ones. Furthermore, we give a level of detail, which, to the best of our knowledge, is not found in the literature, and hence facilitates the implementation of such reductions. Reducing the size of the semidefinite program can speed up computations, and, hence, the reduction we describe is of independent interest.

\subsection{Paper guide}

In Section \ref{sec:examples}, we discuss the necessity of the assumptions made in Theorem \ref{th:gmp_int} and Theorem \ref{th:attainSOS}. We present several counterexamples emphasizing that the relative interiority assumption is crucial to ensure dual attainment. In Section \ref{sec:GMP}, we prove Theorem \ref{th:gmp_int}: duality attainment for the infinite dimensional GMP. In Section \ref{sec:GMP_relax_marshall}, we prove Theorem \ref{th:attainSOS}: duality attainment for each finite dimensional relaxation of the GMP. We do this by adapting results by Marshall, from \cite{marshall_optim_2003,marshall_positive_2008}. In Section \ref{sec:GMP_relax_density}, we use results from \cite{csiszar1975divergence,csiszar2001convex} to find a Slater point for each finite dimensional relaxation of the GMP, which consists of finding a suitable positive measure. We then use this to give another proof of Theorem \ref{th:attainSOS}
 
In Section \ref{sec:apps} we present problems from the literature that admit GMP formulations, and apply
Theorem \ref{th:gmp_int} and Theorem \ref{th:attainSOS} to them. These problems come from tensor optimization, quantum information theory, and quantum optimal transport. First, in Subsection \ref{ssec:sphs}, we prove,  in Corollary \ref{co:attainSOS_sphs}, duality attainment in the special case of $X$ being the product of spheres. In Subsection \ref{ssec:apps_background}, we give background for the application problems, and discuss GMPs over complex variables. In Subsection \ref{ssec:apps_sphs}, we consider problems with GMPs over the product of spheres, which include best rank one approximation of a tensor, the DPS hierarchy, and the quantum Wasserstein distance. Finally, in Section \ref{ssec:apps_nonempty}, we consider positive symmetric tensor decomposition as a GMP instance over the unit ball. We then conclude in Section \ref{sec:conclusion}.

In Appendix \ref{ap:algred}, we present a detailed algebraic reduction of the GMP relaxations. This is used in Appendix \ref{ap:pfofthm} to complete the second proof strategy (started in Subsection \ref{ssec:GMP_relax_density_proof}). In Appendix \ref{ap:spheres_alg}, we show that in the special case of $X$ being the product of spheres, the necessary algebraic assumptions are satisfied.

\subsection*{Acknowledgements}
This work has been supported by the European Union’s HORIZON–MSCA-2023-DN-JD programme under the Horizon Europe (HORIZON) Marie Skłodowska Curie Actions, grant agreement 101120296 (TENORS). The authors are grateful to Fabrice Gamboa for pointing out relevant bibliographic references, and Jonas Britz for helpful discussions regarding the algebraic reduction.

\section{Necessity of assumptions}\label{sec:examples}

We discuss the necessity of the assumptions in our main results, Theorem \ref{th:gmp_int} and Theorem \ref{th:attainSOS}.

\paragraph{Gröbner basis assumption}

Recall that the \emph{ideal} generated by $\p$ is the set $\Ip =  \left\{ \sum_{i=1}^s k_i p_i \mid k_i \in \R[x] \right\}$. Let `$<_{\lex}$' be the lexicographical order on $\R[x]$ such that $x_n <_{\lex} x_{n-1} <_{\lex} \cdots <_{\lex} x_1$. We denote by `$<_{\grlex}$' the graded lexicographical order, and write $x^{\alpha} <_{\grlex} x^{\beta} $ if
\[
|\alpha| < |\beta| \quad \text{or}\quad \Big(|\alpha| = |\beta| \;\; \text{and} \;\; x^\alpha <_{\lex} x^\beta\Big)
\]
for exponent vectors $\alpha,\beta\in \N^n$. Henceforth, we fix the ordering on the monomials as the graded lexicographical order.
Given a polynomial $k = \sum_{\alpha} k_{\alpha} x^{\alpha} \in \R[x]$, the \emph{leading term} of $k$, denoted $\LT(k)$, is the term of $k$ indexed by the maximal $\beta$, with respect to $<_{\grlex}$.
The \emph{leading term ideal} of $\Ip$ is $\LT(\Ip) = \{\LT(k) \mid k \in \Ip \}$. 
A finite subset $\Gb$ of $\Ip$ is called a \emph{Gröbner basis} of $\Ip$ if $\LT(\Ip)=\LT(\Gb)$; that is, if the leading term  of every nonzero polynomial in $\Ip$ is divisible by the leading term of some polynomial in $\Gb$. For more details on Gr\"{o}bner bases, see \cite{cox_ideals_2015}.

As in \eqref{eq:trunc_ideal} we define the truncated ideal $\langle \p \rangle_{t}$ as 
\[
\langle \p \rangle_{t} := \left\{\sum_{l=1}^s k_l p_l \mid \deg(k_l p_l) \leq t\right\} \, .
\]
Note that another natural way to truncate the ideal would be to consider the set $\Ip_{t} := \Ip \cap \R[x]_{t}$. This set contains $\langle \p \rangle_{t}$, but we are only assured equality if the polynomials $p_1,\dots,p_s$ form a Gröbner basis of $\Ip$ with respect to a graded ordering, otherwise the containment may be strict. 

\begin{lemma}\label{le:truncated_ideal}
If the polynomials $p_1,\dots, p_s$ form a Gröbner basis of the ideal $\Ip$ with respect to the graded lexicographic order, then we have the following equality of sets,
\[
\langle \p \rangle_{t} = \left\{\sum_{l=1}^s k_l p_l \mid \deg(k_l p_l) \leq t\right\} = \Ip_t.
\]
\end{lemma}
Lemma \ref{le:truncated_ideal} can be obtained after combining Corollary~2 of \S6 and Lemma~2 of \S9 in \cite[Chapter 2]{cox_ideals_2015}. 

\paragraph{Real radicalness of the ideal}

Given an ideal $\Ip$, its \emph{real vanishing set} is $V_{\R}(\Ip) = \{x \in \R^n \mid k(x)=0, \;\; \text{for all}\; k \in \Ip\}$.
An ideal is called \emph{real radical} if it satisfies $\Ip = \sqrt[\R]{\Ip}$, where
\begin{align*}
\sqrt[\R]{\Ip} & = \{k \in \R[x] \mid k^{2a} + \sigma \in \Ip \text{ for some positive } a \in \N, \sigma \in \Sigma[x]  \}.
\end{align*} 
The assumption of real radicalness is used in the two proofs of Theorem \ref{th:attainSOS}, to show that if $k$ vanishes on $X$, then it is in the ideal $\Ip$, i.e., we have that $k+q = 0$, where $q\in \Ip$. This is precisely the statement of Realnullstellensatz. For more details about real radical ideals, real vanishing sets, and the Realnullstellensatz, see \cite[Section 3]{schweighofer2022realalggeo}.

\paragraph{Ball constraint}
Note that if $X$ is compact, then $g_1= R - \|x\|^2$ can always be fulfilled by simply adding a (possibly redundant) ball inequality constraint in the description of $X$. 
For the sake of simplicity, we assume that there is only one explicit ball constraint for both items of Assumption \ref{hyp:ball}. 
Another possibility would be, as in \cite{tacchi2022convergence}, to add several such constraints. 
Combined with $h_1\equiv1$ and $b_1 > 0$, the ball constraint allows one to ensure that the feasible set of \eqref{eq:gmp_relax} is bounded at each relaxation order. 
Note also that in the first item of Assumption \ref{hyp:ball}, we could suppose that $X$ is described with more than one (not necessarily ball) inequality constraints. 
However, we then go on to assume that the only inequality is, in fact, redundant, since $V_{\R}(\Ip) \subset \big\{x\in \R^n \mid  g_1(x) = R - \|x\|^2 > 0\big\}$. 
Finally, in the case considered in Subsection \ref{ssec:sphs} where $X$ is the product of spheres, there is no need to add any redundant ball constraint as the explicit sphere constraints already ensure that the feasible set of \eqref{eq:gmp_relax} is bounded.

\paragraph{Relative interior assumption} We present examples where duality attainment fails for either the GMP, its relaxations, or both, if the assumption $b\in \relint(K)$ is dropped.

We begin with an example which shows that duality attainment can fail for the infinite dimensional GMP without the assumption $b\in \relint(K)$, even if it holds for the finite dimensional relaxation.

\begin{example}
\label{ex:infno_finyes}
We set  $X=[0,1]$, and consider the following GMP instance and its dual
\begin{equation}\label{eq:counterex}
\begin{aligned}
&\inf_{\mu}  \; \; \int_0^1 x(x-1) \di\mu  \qquad\qquad &&\sup_{\lambda} \; \; \lambda_1 \\
&\st \;  \int_0^1 1 \di\mu = 1  &&\st \; x(x-1) - \lambda_1 - \lambda_2x^2(1-x)\geq 0 \text{ on } [0,1] \\
&\qquad \int_0^1 x^2(1-x) \di\mu = 0 &&\qquad \lambda_1,\lambda_2 \in \R. \\
&\qquad \mu \in\mathcal{M}_+(X) 
\end{aligned}
\end{equation}
In the notation of \eqref{eq:gmp_primal}, we have $f=x(x-1),\; h_1 \equiv 1,\; h_2=x^2(1-x), \; b_1 = 1, \; \text{and}\;\; b_2 = 0$. We set $\delta_c$ as being the Dirac delta measure at $c\in\R$. Feasible measures to the primal must be of the form $\mu = a_0 \delta_0 + a_1 \delta_1$, with $a_0+a_1=1$ and $a_0,a_1\geq0$. This is because $h_2$ is nonnegative on $[0,1]$ and zero only at the endpoints. Hence, the optimal value of the primal problem is clearly zero, which gives an upper bound of zero on the optimal dual value, by weak duality.

This means that, if it were to be feasible, a dual solution of the form $(0,\lambda_2)$ would give the optimal value. However, if $\lambda_1 =0$, then we must find $\lambda_2$ such that 
\[
x^2 -x -\lambda_2x^2 (1-x) \geq 0  \text{ on } [0,1].
\]
This can never be the case as the term $-x$ always dominates the left-hand-side of the above inequality near the origin. 

We now show that there is a sequence of feasible solutions $\{\lambda^{(l)}\}_{l > 0}$, whose objective value tends to zero. 
Set $\lambda_1^{(l)} = -\frac{1}{l}$, and $\lambda_2^{(l)}= \min_{x \in [0,1]} \frac{x^2-x+1/l}{x^2(1-x)}$. 
Note that this latter quantity is bounded from below, as the only points where the denominator vanishes are $0$ and $1$, and the corresponding limits of the expression $ \frac{x^2-x+1/l}{x^2(1-x)}$ are both $+ \infty$. 
Then for all positive integer $l$ one has
\[
x^2 -x + \frac{1}{l} - \lambda_2^{(l)} x^2(1-x)  \geq 0 \text{ on } [0,1]. 
\]
Therefore duality attainment fails for the infinite dimensional GMP relaxation. 

If we now replace the dual inequality constraint by the following equality constraint
\[
x(x-1) - \lambda_1 - \lambda_2x^2(1-x) = \sigma_0(x) + \sigma_1 x (x-1),
\]
with a sum-of-squares polynomial $\sigma_0$ of degree at most 2 and a nonnegative scalar $\sigma_1$, then necessarily $\lambda_2 = 0$ since the right-hand-side has degree at most 2. Thus one has 
\[
(1+\sigma_1) x(x-1) - \lambda_1 = \sigma_0(x),
\]
which is equivalent to the nonnegativity of $(1+\sigma_1) x(x-1) - \lambda_1$ since nonnegative univariate polynomials are sums of squares. 
Since the minimum of $x(x-1)$ over $\R$ is $-1/4$ we obtain the optimal solution $\lambda_1=-1/4, \lambda_2=0, \sigma_0=(x-1/2)^2, \sigma_1=0$, yielding attainment of the finite dimensional sum-of-squares program. 
\end{example}

The following example shows that duality attainment can fail for the finite dimensional GMP relaxations without the assumption $b\in \relint(K)$, even if duality attainment holds for the infinite dimensional GMP.

\begin{example}\label{ex:infyes_finno}
We set  $X=[0,1]$, and consider the following GMP instance and its dual
\begin{equation}\label{eq:counterex2}
\begin{aligned}
&\inf_{\mu}  \; \; \int_0^1 x \di\mu  \qquad\qquad &&\sup_{\lambda} \; \; \lambda_1 \\
&\st \;  \int_0^1 1 \di\mu = 1  &&\st \; x - \lambda_1 - \lambda_2x^2\geq 0 \text{ on } [0,1] \\
&\qquad \int_0^1 x^2 \di\mu = 0 &&\qquad \lambda_1,\lambda_2 \in \R. \\
&\qquad \mu \in\mathcal{M}_+(X) 
\end{aligned}
\end{equation}
In the notation of \eqref{eq:gmp_primal}, we have $X=[0,1],\; f=x,\; h_1 \equiv 1,\; h_2=x^2, \; b_1 = 1, \; \text{and}\;\; b_2 = 0$. Finally, suppose $X$ is given with the semialgebraic description $X=\{x\in \R\;|\; g(x)=x^3(1-x)\geq 0\}$. The only feasible solution to the primal is the Dirac delta measure at zero, i.e., $\mu = \delta_0$. This is because $h_2$ is nonnegative on $[0,1]$ and zero only at the $x=0$. Hence, the optimal value of the primal problem is zero, and is attained. Furthermore, it gives an upper bound of zero on the optimal dual value. Thus, the optimal value of the infinite dual problem is attained by the solution $\lambda_1,\lambda_2=0$. This shows duality attainment holds for the infinite dimensional GMP.

With $t_g=\lceil\frac{\deg g}{2} \rceil$, the $t^{th}$ primal relaxation and dual strengthening are given by
\[
\begin{aligned}
&\inf_{\mom}  \; \; L_{\mom}(x)  \qquad\qquad &&\sup_{\lambda, \sigma_0,\sigma_1}  \; \; \lambda_1 \\
&\st \;  L_{\mom}(1) = 1 && \st \; x- \lambda_1 - \lambda_2x^2 = \sigma_0 + \sigma_1 g \\
&\qquad L_{\mom}(x^2) = 0  && \qquad \sigma_0 \in \Sigma[x]_{2t}, \; \sigma_1 \in \Sigma[x]_{2(t-t_g)}\\
&\qquad M_t(\mom)\succeq 0,\; M_{t-t_g}(g \mom) \succeq 0 && \qquad \lambda_1,\lambda_2 \in \R\\
& \qquad \mom \in \R^{\N^n_{2t}}. &&
\end{aligned}
\]
The solution $\mu=\delta_0$ of the infinite dimensional problem gives a solution to each level of the primal relaxation with value zero, which acts as an upper bound on the optimal value of the dual strengthening. This means that, if it were to be feasible, a dual solution of the form 
$(0,\lambda_2)$ would give the optimal value. However, if $\lambda_1 =0$, then we must find $\lambda_2$ such that 
\[
x -\lambda_2 x^2 = \sigma_0 + \sigma_1 x^3(1-x).
\]
This can never be the case since the linear part of the left hand side is $x+0$, while the linear term on the right hand side will be contained in $\sigma_0$. However, $\sigma_0$ is nonnegative around zero, so it cannot have a linear part of $x+0$.

We now show that there is a sequence of feasible solutions $\{\lambda^{(l)}\}_{l > 0}$, whose objective value tends to zero, which implies that duality attainment fails for the finite dimensional GMP relaxation. Set $\lambda^{(l)}_1 = -\frac{1}{4l^2},  \lambda_2^{(l)}=-l^2$, and write
\[
x + \frac{1}{4l^2} + l^2 x^2= \left(l x + \frac{1}{2l}\right)^2 = \sigma_0^{(l)}(x) + \sigma_1^{(l)}(x) g , 
\]
with $\sigma_0^{(l)}= \left(l x + \frac{1}{2l}\right)^2$ and $\sigma_1^{(l)} =0$.
\end{example}

Finally, the next example shows duality attainment can fail for both the finite dimensional GMP relaxation and the infinite dimensional GMP, without the assumption $b\in \relint(K)$.

\begin{example}\label{ex:infno_finno}
We use the same setup as in Example \ref{ex:infyes_finno}, with the only change being $f=-x$. In this case, dual attainment fails for the infinite dimensional GMP. This is because there does not exist $\lambda_2$ such that 
\[
-x - \lambda_2 x^2 \geq 0 \text{ on } [0, 1], 
\]
since the term $-x$ always dominates the left-hand-side of the above inequality near the origin. 
Similarly to Example \ref{ex:infyes_finno}, duality attainment fails for the finite dimensional GMP relaxation: set $\lambda^{(l)}_1 = - \frac{1}{4l^2},  \lambda_2^{(l)}=-l^2$, and write
\[
-x + \frac{1}{4l^2} + l^2 x^2= \left(l x - \frac{1}{2l}\right)^2 = \sigma_0^{(l)}(x) + \sigma_1^{(l)}(x) g , 
\]
with $\sigma_0^{(l)}= \left(l x - \frac{1}{2l}\right)^2$ and $\sigma_1^{(l)} =0$.
\end{example}

\section{Duality attainment of the GMP}\label{sec:GMP}

In this section, we prove duality attainment for \eqref{eq:gmp_primal}, under the relative interior assumption.

Recall that for a convex set $C$, its \emph{relative interior}, denoted $\relint(C)$, is the interior of $C$ within its affine hull, denoted $\aff(C)$,
\[
\relint(C) = \{b \in \aff(C) \mid  B(b,\varepsilon)  \cap \aff(C) \subset C \text{ for some } \varepsilon > 0 \},
\]
where $B(b,\varepsilon) = \{b' \in \R^m \colon \|b - b'\|_{\infty} \le \varepsilon \}$ is the ball with radius $\varepsilon$ centered on $b$. When $C$ contains the origin, then one can replace $\aff(C)$ by the linear hull of $C$, denoted by $\lin(C)$. For more information about the relative interior, see \cite[Section 6]{rockafellar1997convex}.

Given a polynomial map $h = (h_1,\dots,h_m)$ with $h_1,\dots,h_m \in \R[x]$, the so-called \emph{pushforward measure} of a given $\mu \in \mathcal{M}_+(X)$ under $h$ is
defined as 
\begin{align*}
h_{\#} \mu (C) = \mu (h^{-1}(C)) = \mu (\{ x \in X : h(x) \in C \}),
\end{align*}
for every Borel set $C \in \mathscr{B}(\R^m)$. In this way, $h_{\#} \mu$ is a nonnegative Borel measure on $\R^m$. The change of variable formula then gives
\[
\int_{\R^m} g  \, \di h_{\#}\mu = \int_X g \circ h \, \di \mu,
\]
for every Borel function $g \colon \R^m \to \R$ for which the right hand side of the equation exists.

The following proposition gives an equivalent description of the set $K$, as in \eqref{eq:K_cone}. In this proposition, the \emph{conic hull} of a set $Y \subseteq \R^m$ is the smallest pointed convex cone that contains $Y$, i.e., the set
\[
\{\alpha_1 y^{(1)} + \dots + \alpha_N y^{(N)} \mid N \ge 1, \ \alpha_1, \dots, \alpha_N \ge 0, \ y^{(1)},\dots,y^{(N)} \in Y\} \, .
\]
\begin{proposition}\label{pr:tchakaloff}
Consider the polynomial map $h$ and the cone $K$ defined in \eqref{eq:K_cone}. 
The cone $K$ is the conic hull of $h(X)$.
\end{proposition}
\begin{proof}
Denote $Z := h(X)$. Every element in the conic hull of $Z$ can be written as $\sum_{j=1}^N \alpha_j h(x^{(j)})$ for some $x^{(j)} \in X$ and $\alpha_j \ge 0$. By selecting $\nu = \sum_{j=1}^N \alpha_j \delta_{x^{(j)}} \in \mathcal{M}_+(X)$, this conic hull element belongs to $K$.

Conversely, fix $\nu \in \mathcal{M}_+(X)$. If $\nu$ is the zero measure, then the vector $(\int_X h_i \di \nu)_i$ is the zero vector, which belongs to the conic hull of $Z$. If $\nu$ is not the zero measure, let $\eta := h_{\#} \nu$ be the pushforward of $\nu$. Note that $\eta$ has finite first moments, because the change of variable formula and compactness of $X$ give
\[
\int_{\R^m} \| y\| \di \eta(y) = \int_{X} \| h(x)\| \di \nu(x) < +\infty \, .
\]
Hence, Tchakaloff's theorem \cite{tchakaloff1957cubature} in the form stated in \cite[Corollary 2]{BayerTeichmann:2006} gives that there exists a cubature rule for the integration of the polynomials $h_i$ over $X$ with respect to $\nu$. In other words, there exists $1 \le N \le m$, points $x^{(1)}, \dots, x^{(N)} \in X$, and weights $\alpha_1, \dots, \alpha_N > 0$ such that the equality $\int_X h_i \di \nu = \sum_{j} \alpha_j h_i(x^{(j)})$ holds for all $i$. Hence, the vector $(\int_X h_i \di \nu)_i$ belongs to the conic hull of $Z$.
\end{proof}
The dual cone of $K$ is the cone of all nonnegative functionals on $K$, i.e., $K^* := \{\lambda \in \R^m \mid \sum_{i=1}^m \lambda_i b_i \ge 0 \text{ for all $b \in K$}\}$. By Proposition~\ref{pr:tchakaloff}, we have the following description.
\begin{corollary}\label{cor:dual_cone}
We have $K^* = \{\lambda \in \R^m \mid \sum_{i = 1}^{m} \lambda_i h_i \ge 0 \text{ on $X$}\}$.
\end{corollary}
We next recall that when $X$ is compact, the primal infinite dimensional linear problem \eqref{eq:gmp_primal} with polynomial input data is attained if it has a feasible solution. The following statement is an adaptation of \cite[Theorem~1]{klerk2019survey}.
\begin{theorem}\label{th:nogap_gmp}
Assume that \eqref{eq:gmp_primal} has a feasible solution. 
Then there is no duality gap between \eqref{eq:gmp_primal} and \eqref{eq:gmp_dual}. Furthermore, \eqref{eq:gmp_primal} has an optimal solution. 
\end{theorem}
\begin{proof}
Let us consider the polynomial map $\Phi = (f,h_1,\dots,h_m)$. 
Since $\Phi$ is continuous and $X$ is compact, the image set $\Phi(X)$ is closed. 
As in the proof of Proposition \ref{pr:tchakaloff}, the cone 
\[
\hat{K} = \left\{\left(\int_X f \di \mu, \int_X h_1 \di \mu,\dots,\int_X h_m \di \mu  \right)  \mid \mu \in \mathcal{M}_+(X) \right\}
\] 
is the conic hull of $\Phi(X)$, therefore it is closed. 
The value of \eqref{eq:gmp_primal} is bounded from below as $f$ is a polynomial and $X$ is compact. 
The result then follows by the duality theory of conic linear optimization (see, e.g., \cite[Section~IV.7.2]{barvinok2025course} or \cite[Theorem~1]{klerk2019survey}). 
\end{proof}

\begin{theorem*}[Repetition of Theorem \ref{th:gmp_int}]
Assume that $h_1\equiv1$, $b_1 > 0$ and that $b = (b_1,\dots,b_m)$ is in the relative interior of $K$. 
Then the dual program \eqref{eq:gmp_dual} attains its optimal value.
\end{theorem*}
\begin{proof}
Let $f_{\min} = \min_{x \in X}f(x)$ and $f_{\max} = \max_{x \in X} f(x)$. 
Since $b$ is in the relative interior of $K$, there exists $\varepsilon > 0$ such that the set
\[
U := \{b' \in \R^m \colon \|b - b'\|_{\infty} \le \varepsilon \} \cap \lin(K)
\]
is included in $K$. Note that $\varepsilon$ only depends on $b$. 
By Theorem \ref{th:nogap_gmp} there is no duality gap between \eqref{eq:gmp_primal} and \eqref{eq:gmp_dual}. 
Since the primal optimal value is greater or equal to $\int_X f_{\min} \di \mu = b_1 f_{\min}$, let us consider any feasible solution $\lambda=(\lambda_1,\dots,\lambda_m)$ with $\sum_{i=1}^m b_i \lambda_i \geq b_1 f_{\min} - \eta$ for some $\eta > 0$. Without loss of generality, one can further restrict the decision variable $\lambda$ of  \eqref{eq:gmp_dual} to belong to $\lin(K)$. 
Indeed, write $\lambda = u+ v$, with $u \in \lin(K)$ and $v \in \lin(K)^\perp$. Then $\sum_{i=1}^m b_i \lambda_i = \sum_{i=1}^m b_i u_i + \sum_{i=1}^m b_i v_i = \sum_{i=1}^m b_i u_i$, since $b \in \lin(K)$. Thus, the dual objective function is unchanged. 
In addition, Corollary~\ref{cor:dual_cone} gives 
\[
\lin(K)^\perp = K^* \cap -K^* = \left\{v \in \R^m: \sum_{i=1}^m v_i h_i(x)=0 \text{ on } X \right\} \, .
\]
Hence, we have $\sum_{i=1}^m  h_i(x)\lambda_i = \sum_{i=1}^m  h_i(x)u_i$, for all $x \in X$. 
Since the objective function and inequality constraint remain both unchanged in \eqref{eq:gmp_dual}, let us assume in the sequel that $\lambda \in \lin(K)$. 

Let $T := \|\lambda\|$ and $\lambda' := \lambda/T$. 
We prove that $T$ is upper bounded by a finite positive number. To do so, consider $b' := b + \varepsilon \lambda'$. Since $\lambda$ is in $\lin(K)$, we get $b' \in U$ and 
\[
\sum_{i=1}^m b'_i \lambda'_i - \frac{b_1 f_{\min}-\eta}{T} = \sum_{i=1}^m b_i \lambda'_i - \frac{b_1 f_{\min}-\eta}{T} + \varepsilon \ge \varepsilon \, .
\]
Since $b'$ is in $K$, there exist weights $w_j > 0$ and points $x^{(j)} \in X$ such that $b'_i = \sum_j w_j h_i(x^{(j)})$ for all $i=1,\dots,m$. 
Hence,
\begin{align*}
\sum_{j}\frac{w_j}{\sum_k w_k}\left(\sum_{i=1}^m h_i(x^{(j)})\lambda'_i - \frac{b_1 f_{\min}-\eta}{T \sum_k w_k}\right) 
& = \frac{1}{\sum_k w_k} \left( \sum_{i=1}^m b_i' \lambda'_i - \frac{b_1 f_{\min}-\eta}{T} \right) \\
& \ge \frac{\varepsilon}{\sum_k w_k}.
\end{align*}
Therefore, there exists $x^{(j)} \in X$ such that 
\[\sum_{i=1}^m h_i(x^{(j)})\lambda'_i - \frac{b_1 f_{\min}-\eta}{T \sum_k w_k} \ge \frac{\varepsilon}{\sum_k w_k}.\]
By feasibility of $\lambda$ we get 
\[\sum_{i=1}^m h_i(x^{(j)}) \lambda'_i = \frac{1}{T}\sum_{i=1}^m h_i(x^{(j)}) \lambda_i \le \frac{1}{T}f(x^{(j)}) \le \frac{f_{\max}} {T},\] and 
\[\frac{\varepsilon}{\sum_k w_k} + \frac{b_1 f_{\min}-\eta}{T \sum_k w_k} \leq \frac{f_{\max}}{T}.\]
Note that $b'_1 = \sum_k w_k$ and $|b'_1 - b_1| = \varepsilon |\lambda'_1| \le \varepsilon$, so $\sum_k w_k \le b_1 + \varepsilon$. 
Therefore 
\begin{align*}
T \le \frac{\sum_k w_k}{\varepsilon}\left(f_{\max} - \frac{b_1 f_{\min}-\eta}{\sum_k w_k} \right)  
& = \frac{\sum_k w_k f_{\max} - (b_1 f_{\min}-\eta)}{\varepsilon} \\
& \leq  \frac{(b_1+\varepsilon)f_{\max} - b_1 f_{\min}+\eta}{\varepsilon}, 
\end{align*}
which does not depend on the choice of $\lambda$.

The feasible set of \eqref{eq:gmp_dual} is the intersection of infinitely many polynomial superlevel sets, and, thus, is closed. This concludes the proof. 
\end{proof}

\begin{remark}
\label{rk:gmp_int}
Under the same relative interiority assumption, Theorem \ref{th:gmp_int} has been  proved by Nie in the special case where $K$ is the truncated moment cone, see \cite[Proposition 3.6 (ii)]{nie_atruncKprob_2014}. 
To emphasize this relation, we first mention an equivalent characterization of relative interior, that can be found in \cite[Theorem 2]{luo1997duality}:
\[
\relint(K) = \left\{b \in K \mid  \sum_{i=1}^m b_i \lambda_i > 0, \; \forall \lambda \in K^* \backslash \lin(K)^\perp  \right\},
\]
with $\lin(K)^\perp = K^* \cap - K^*$. 
If $K$ is the truncated moment cone, namely
\[
K = \left\{z = \left(\int_X x^{\alpha} \di \mu\right)_{\alpha\in\N^n_d} :  \mu \in \mathcal{M}_+(X) \right\},
\]
then $K^*$ and $\lin(K)^\perp$ correspond to coefficient vectors of polynomials being nonnegative on $X$ and vanishing on $X$, respectively. 
Given $z$ in the truncated moment cone $K$, it is assumed in \cite[Proposition 3.6 (ii)]{nie_atruncKprob_2014} that $\sum_{\alpha \in \N^n_d} p_{\alpha} z_{\alpha} > 0$ for all nonzero $p \in \R[x]_d$ being nonnegative on $X$. 
As in the proof of Theorem \ref{th:gmp_int}, the proof of \cite[Proposition 3.6 (ii)]{nie_atruncKprob_2014} restricts the feasible set of \eqref{eq:gmp_dual} to $\lin(K)$, which is equivalent to assume that $z \in \relint(K)$.

Nie's proof requires to show that the vector $z$ is a subsequence of another sequence $z'$ that has a representing measure, i.e., $z'$ is the moment sequence of a probability measure; see \cite[Proposition 3.3]{nie_atruncKprob_2014}.  
By contrast with Nie's construction, our proof is rather short and elementary, and can be applied to any generalized moment cones built from polynomial maps. 
\end{remark}

\section{Duality attainment of relaxations 1:\texorpdfstring{\\ Closed $\QM_{2t}$}{}}\label{sec:GMP_relax_marshall}

In this section, we prove that duality attainment holds for the GMP relaxation \eqref{eq:gmp_relax}, for every level $t$. We show that the feasible set of the dual program \eqref{eq:dual_relax} is closed in the Euclidean topology in $\R[x]_{2t}$.

\begin{proposition}\label{prop:moduleclosed}
Let $X$ be a compact subset of $\R^n$ and suppose Assumption \ref{hyp:ball} holds, then the following statements hold. 
\begin{enumerate}[(i)]
\item If $\Ip$ is a real radical ideal and $\p=\{p_1,\dots,p_s\}$ forms a Gröbner basis of $\Ip$ with respect to the graded lexicographical order, then the set $\QM_{2t}(g_1) + \langle \p \rangle_{2t}$ is closed for all $t \geq t_{\min}$.
\item If $\p=\{0\}$ and $X$ has nonempty interior, then the truncated quadratic module $\QM_{2t}(\g)$ is closed for all $t \geq t_{\min}$. 
\end{enumerate}
\end{proposition}
\begin{proof}
\textit{(i)} We suppose that the variety is not empty, and follow a similar reasoning as for the proof of \cite[Lemma 4.1.4]{marshall_positive_2008}.
First, since $\p=\{p_1,\dots,p_s\}$ forms a Gröbner basis  of $\Ip$ with respect to the graded lexicographical order, one has, by Lemma \ref{le:truncated_ideal},
\[
\Ip \cap \R[x]_{2t} = \Ip_{2t}  = \langle \p \rangle_{2t} = \left\lbrace\sum_{l=1}^s k_i p_i  \mid k_i \in \R[x], \deg(k_i p_i) \leq 2t \right\rbrace.
\]
Let $\mathcal{B}_t$ be a basis of the quotient vector space $\R[x]_t/\Ip_t$; see Appendix \ref{ap:algred} for the notation. 
Let $\Sigma'[x]_{2t}$ be the set of sums of squares of polynomials supported on $\mathcal{B}_t$. 
We define $\Sigma'[x]_{2(t-1)}$ similarly.

For each $\sigma = \sum_i \phi_i^2 \in \Sigma[x]_{2t}$, one has $\phi_i \in \Span_{\R}(\mathcal{B}_{t}) \oplus \Ip_{t}$. 
Writing $\phi_i = \phi_i' + y_i$ with $\phi_i' \in \Span_{\R}(\mathcal{B}_{t})$ and $ y_i \in \Ip_{t}$, we have $\sigma  = \sum_i (\phi_i'^2 + 2\phi_i'y_i + y_i^2) \in \Sigma'[x]_{2t} + \Ip_{2t}$. 
This implies that $\QM_{2t}(g_1) + \langle \p \rangle_{2t} = \QM'_{2t}(g_1) + \langle \p \rangle_{2t}$, where
\[
\QM'_{2t}(g_1) = \left\{\sigma_0 + \sigma_1 g_1 \mid \sigma_0\in \Sigma'[x]_{2t},\; \sigma_1 \in \Sigma'[x]_{2(t-1)}   \right\}.
\]
Any $\sigma \in \Sigma'[x]_{2t}$ can be written as $v_t^T G v_t$ with $v_t= \{q_i(x)\}_{q_i \in \mathcal{B}_t}$ for some PSD matrix $G$ of size $|\mathcal{B}_t| = r_t$. 
As the set of $r_t \times r_t$ PSD matrices is closed the set $\Sigma'[x]_{2t}$ is also closed.
Since $\Ip$ is real radical, $\Ip_{2t} = \Ip \cap \R[x]_{2t}$ is exactly the set of polynomials of degree at most $2t$ that vanish on $V_{\R}(\Ip)$. Hence, the set $\Ip_{2t} = \langle \p \rangle_{2t}$ is closed. 

We will now prove that $\QM'_{2t}(g_1) + \langle \p \rangle_{2t}$ is closed as well. For this we have to show that if a sequence $\{F^{(l)}\}_l$, defined by 
\[
F^{(l)} = \sigma_0^{(l)} + \sigma_1^{(l)} g_1 + \sum_{i=1}^s k_i^{(l)} p_i,
\]
with $\{\sigma_0^{(l)}\}_l \subset \Sigma'[x]_{2t}$, $\{\sigma_1^{(l)}\}_l \subset \Sigma'[x]_{2t-2}$, and $\{k_i^{(l)}\}_l \subset \R[x]_{2t - \deg p_i}$, converges to $F \in \R[x]_{2t}$, then $F \in \QM'_{2t}(g_1) + \langle \p \rangle_{2t}$. 
As above we write $\sigma_1^{(l)} = v_{t-1}^T G_1^{(l)} v_{t-1}$, and denote by $\|G^{(l)}\|$ the Euclidean norm of $G^{(l)}=(G_0^{(l)},G_1^{(l)})$, viewing $G^{(l)}$ as a tuple of real numbers.
If we can show that $\|G^{(l)}\|$ is bounded, then we are done. 

Suppose by contradiction that the sequence of norms $\{\|G^{(l)}\|\}_l$ is not bounded. Replacing each $\{\sigma_j^{(l)}\}_l$ by a suitable subsequence, assume that $\|G^{(l)}\| \to \infty$ and also, replacing $\{\sigma_j^{(l)}\}_l$ by a subsequence again, assume that $\frac{G^{(l)}}{\|G^{(l)}\|} \to \tilde{G} = (\tilde{G_0},\tilde{G_1})$, for some PSD matrices $\tilde{G_0}$ and $\tilde{G_1}$ of sizes $r_t$ and $r_{t-1}$, respectively,  with $\|\tilde{G}\|=1$. 
Then
\[
\frac{1}{\|G^{(l)}\|}\left(\sigma_0^{(l)} + \sigma_1^{(l)} g_1 + \sum_{i=1}^s k_i^{(l)} p_i\right) = \frac{1}{\|G^{(l)}\|} F^{(l)} \to 0.
\]
Since the sets $\Sigma'[x]_{2t}$ and $\Ip_{2t}$ are both closed, we also obtain that 
\begin{align*}
\frac{1}{\|G^{(l)}\|}\sigma_0^{(l)} \to \tilde{\sigma_0} \in \Sigma'[x]_{2t},\;\; \frac{1}{\|G^{(l)}\|}\sigma_1^{(l)} \to \tilde{\sigma_1} \in \Sigma'[x]_{2(t-1)} \\
\text{and}\;\;\frac{1}{\|G^{(l)}\|} \sum_{i=1}^s k_i^{(l)} p_i \to \sum_{i=1}^s \tilde{k_i} p_i \in \Ip_{2t}.
\end{align*}
Hence $\tilde{\sigma}_0 + \tilde{\sigma_1} g_1 \in \Ip_{2t}$, and the sum vanishes on $V_{\R}(\Ip)$. Since each summand is nonnegative on $V_{\R}(\Ip)$, we have that each summand must also vanish on $V_{\R}(\Ip)$. 
Furthermore, since $g_1$ is positive on $V_{\R}(\Ip)$, we obtain that $\tilde{\sigma_1}$ vanishes on $V_{\R}(\Ip)$. 
Writing $\tilde{\sigma_0} = \sum_{i} \phi_i'^2$, the above implies that $\phi'_i$ vanishes on $V_{\R}(\Ip)$. 
Since $\Ip$ is real radical, $\phi'_i \in \Ip_t$.   
But one also has that $\phi'_i \in \Span_{\R}(\mathcal{B}_{t})$, so  $\phi'_i = 0$. 
Thus $\tilde{\sigma_0}=0$ and $\tilde{G_0} = 0$. 
Similarly we obtain $\tilde{\sigma_1}=0$ and $\tilde{G_1} = 0$, 
yielding a contradiction with $\|\tilde{G}\|=1$. 

\textit{(ii)} We consider a variant of the above proof in the case when the variety is all of $\R^n$. 
In this case, we obtain that $\tilde{\sigma} + \sum_{j=1}^{\ell} \tilde{\sigma_j} g_j = 0$. 
Since this latter sum vanishes in particular on $X$ and each summand is nonnegative on $X$ then each summand must also vanish on $X$. 
Since $X$ has nonempty interior there exist $x' \in X$, $\varepsilon > 0$ and a ball $B(x',\varepsilon)$ such that $g_j(x) > 0$ for all  $x \in B(x',\varepsilon)$, implying that each $\tilde{\sigma_j}$ vanishes on $B(x',\varepsilon)$, so $\tilde{\sigma_j}=0$. 
\end{proof}

\begin{theorem}[Repetition of Theorem \ref{th:attainSOS}]
\label{th:attainSOS1}
Let $X$ be a compact subset of $\R^n$ satisfying Assumption \ref{hyp:ball}. 
Suppose that either
\begin{enumerate}[(i)]
    \item The ideal $\Ip=(p_1,\dots,p_s)$ is real radical, and that $\p=\{p_1,\dots,p_s\}$ forms a Gröbner basis of $\Ip$ with respect to the graded lexicographic order, or
    \item $\p=\{0\}$ and $X$ has nonempty interior. 
\end{enumerate}
Assume further, that $h_1\equiv1$, $b_1 > 0$, and that $b = (b_1,\dots,b_m)$ is in the relative interior of $K$. 
Then the semidefinite program \eqref{eq:dual_relax} attains its optimal value for all $t \geq t_{\min}$.
\end{theorem}
\begin{proof}
The proof follows the same strategy as the one of Theorem \ref{th:gmp_int}. 
One shows that the value $\primal_t$ of the primal program \eqref{eq:gmp_relax} is lower bounded by leveraging Assumption \ref{hyp:ball} together with the fact that $h_1 \equiv 1$ and $b_1 >0$. Under these assumptions, one can show that all entries of any feasible $z$ for \eqref{eq:gmp_relax} are uniformly bounded in absolute value; see, e.g.,  \cite[Lemma 5]{tacchi2022convergence}. 

Since $\primal_t$ is lower bounded, there must exist a feasible finite $\lambda = (\lambda_1,\dots,\lambda_m)$ for \eqref{eq:dual_relax}. 
By Proposition \ref{prop:nogaprelax}, there is no duality gap between \eqref{eq:gmp_relax} and \eqref{eq:dual_relax}. 
So let us consider any feasible solution $\lambda$ with $\sum_{i=1}^m b_i \lambda_i \geq \primal_t - \eta$ for some $\eta > 0$. 
Without loss of generality, one can further restrict the decision variable $\lambda$ of \eqref{eq:dual_relax} to belong to $\lin(K)$. 
Indeed, write $\lambda = u+ v$, with $u \in \lin(K)$ and $v \in \lin(K)^\perp$. 
Then exactly as in the proof of Theorem \ref{th:gmp_int} the objective function is unchanged. 
In addition, one has
\[
\lin(K)^\perp = K^* \cap -K^* = \left\{v \in \R^m: \sum_{i=1}^m v_i h_i(x)=0 \text{ on } X \right\}.
\]

In the case (i), one has that $\sum_{i=1}^m v_i h_i$ belongs to $\Ip=(p_1,\dots,p_s)$, since $\Ip$ is real radical. 
Since $\p=\{p_1,\dots,p_s\}$ form a Gröbner basis of $\Ip$ with respect to the graded lexicographic order, Lemma \ref{le:truncated_ideal} gives that the polynomial $\sum_{i=1}^m v_i h_i \in \R[x]_{2t} \cap \Ip$ can be written as $\sum_{l=1}^s k_l' p_l$ with $\deg k_l' \leq 2 t - \deg p_l $. 
Therefore one can replace the constraint $f- \sum_{i=1}^m \lambda_i h_i \in \QM_{2t}(g_1) +  \langle \p \rangle_{2t}$ by the constraint $f- \sum_{i=1}^m u_i h_i \in \QM_{2t}(g_1) +  \langle \p \rangle_{2t}$.

In case (ii), one has that $\sum_{i=1}^m v_i h_i$ is the zero polynomial, since $X$ has nonempty interior. 
Therefore $\sum_{i=1}^m \lambda_i h_i = \sum_{i=1}^m u_i h_i$, and one can replace the constraint $f- \sum_{i=1}^m \lambda_i h_i \in \QM_{2t}(\g)$ by the constraint $f - \sum_{i=1}^m u_i h_i \in \QM_{2t}(\g)$.
Overall, the objective function and inequality constraint remain both unchanged in \eqref{eq:dual_relax}, so let us assume in the sequel that $\lambda \in \lin(K)$. 

Let $T_t := \|\lambda\|$ and $\lambda' := \lambda/T_t$. 
As in the proof of Theorem \ref{th:gmp_int}, one obtains the upper bound $T_t \leq \frac{(b_1+\varepsilon)f_{\max} -  \primal_t+\eta}{\varepsilon}$, proving that $T$ is upper bounded by a finite positive number. 
The desired result follows from Proposition \ref{prop:moduleclosed}. 
\end{proof}

\section{Duality attainment of relaxations 2: \texorpdfstring{\\ strict feasibility}{}}\label{sec:GMP_relax_density}

In this section, we present a different proof of the duality attainment of the GMP relaxation \eqref{eq:gmp_relax}, for every level $t$. We do this by using the strong duality of semidefinite programs, which we state in the following theorem (adapted in our context); see, e.g.,~\cite[Theorem 5.4]{theobald_book_2024} or \cite[Section 2.1.5]{blekherman_convexalggeo_2013} for more details.
\begin{theorem}[Strong Duality of Semidefinite Programs]\label{th:strong_duality}
If the primal program \eqref{eq:gmp_relax} is bounded and has a strictly feasible solution, then the dual program \eqref{eq:dual_relax} attains its optimal value and there is no duality gap.
\end{theorem}
Note that some sources refer to strong duality as the primal and dual having no duality gap. Recall that there is no duality gap between \eqref{eq:gmp_relax}
and \eqref{eq:dual_relax}, by Proposition \ref{prop:nogaprelax}. Hence, we show that, under the assumption $b\in\relint(K)$, there exists a measure which gives a strictly feasible solution for every level of \eqref{eq:gmp_relax}. Finding this measure is the content of Subsection \ref{ssec:msr}.

We then go on to prove dual attainment under Assumption \ref{hyp:ball}, like we did in Section \ref{sec:GMP_relax_marshall}, in two cases: when $X$ has empty and nonempty interior, i.e., when $\p\neq\{0\}$ and $\p=\{0\}$, respectively. 
Some extra work is required to show the former case, the technical details of which are in Appendices \ref{ap:algred} and \ref{ap:pfofthm}.

\subsection{Existence of a strictly feasible measure}\label{ssec:msr}

Suppose $X$ is a nonempty compact subset of $\R^n$. In this subsection, we prove that, under a relative interior assumption, we can find a positive measure $\nu$ which is feasible for \eqref{eq:gmp_primal}, and such that
\[
\int_X g \,\di \nu = 0 \iff g(x) = 0 \text{ for all $x \in X$},
\]
for any nonnegative continuous function $g \colon X \to \R_{\ge 0}$. This will then be used in the next subsection to get a Slater point for each pseudo-moment program of the associated hierarchy of semidefinite relaxations. Our proof relies on results from \cite{csiszar1975divergence} about the existence and characterization of the $\Idiv$-projection, see also \cite[Chapter X]{Nagasawa:1993} for more discussion.

Given two measures $\nu,\mu \in \mathcal{M}_+(X)$,  the notation $\nu \ll \mu$ means that $\nu$ is absolutely continuous with respect to $\mu$, that is, for every $B \in \mathscr{B}(X)$, $\mu(B) = 0$ implies $\nu(B) = 0$. 
If $\nu \ll \mu$, the corresponding density (Radon--Nikodym derivative) is denoted by $\frac{\di \nu}{\di \mu}$. 
The $\Idiv$-divergence or  Kullback-Leibler information number $\Idiv(\nu||\mu)$ is defined as 
\[
\Idiv(\nu || \mu)=
\begin{cases}
\int_X \log \big(\frac{\di \nu}{ \di \mu}\big) \di \nu = \int_X \frac{\di \nu}{ \di \mu} \log \big(\frac{\di \nu}{ \di \mu}\big) \di \mu, \;\text{ if } \nu \ll \mu, \\
+ \infty \;\text{ otherwise}.
\end{cases}
\]
For a modern exposition on information theory and relative entropy, we refer the interested reader to \cite[Chapter~5]{gray2011entropy}. Let us fix $\mu \in \mathcal{M}_+(X)$ and a convex subset $\mathcal{E}$ of $\mathcal{M}_+(X)$. If each $\nu \in \mathcal{E}$ satisfies  $\Idiv(\nu || \mu) < +\infty$, then a measure $\nu_{\proj}$ satisfying
\[
\Idiv(\nu_{\proj} || \mu) = \inf_{\nu \in \mathcal{E}} \Idiv(\nu || \mu), 
\]
is called the \emph{$\Idiv$-projection} of $\mu$ on to $\mathcal{E}$. 

Given a probability measure $\mu \in \mathcal{P}(X)$, we fix polynomials $h_1,\dots,h_m\in\R[x]$, and define 
\[
K_{\mu} := \left\{ \left(\int_X h_i \di \nu\right)_{i=1,\dots,m} :  \nu \in \mathcal{P}(X), \ \Idiv(\nu || \mu) < +\infty \right\} \, .
\]
The set $K_{\mu} \subseteq \R^m$ is convex by construction. 
Moreover, given $b \in K_{\mu}$ we denote by $\mathcal{E}(b) \subseteq \mathcal{P}(X)$ the set
\[
\mathcal{E}(b) := \left\{\nu \in \mathcal{P}(X) \colon \int_X h_i \di \nu = b_i,\;\; \text{for} \; i =1,\dots,m\right\} \, .
\]
The set $\mathcal{E}(b)$ is a convex set of probability measures. Furthermore, note that, by definition, if $b \in K_{\mu}$, then the set $\mathcal{E}(b)$ is nonempty and contains a probability measure $\nu$ such that $\Idiv(\nu || \mu) < +\infty$. Since $X$ is compact, the functions $h_i$ are bounded on $X$. In particular, if we equip $\mathcal{P}(X)$ with the topology of the total variation distance $d_{\mathrm{TV}}(\mu,\nu) = \sup_{A \in \mathcal{B}(X)}|\mu(A) - \nu(A)|$, then $\mathcal{E}(b)$ is closed by \cite[Theorem~2.1]{Billingsley:1999}. 
Hence, the following theorem follows by combining \cite[Theorems 2.1 and 3.3]{csiszar1975divergence}, see also \cite[Theorem 10.2]{Nagasawa:1993}.

\begin{theorem}\label{th:idivergence}
If $b \in K_{\mu}$, then the $\Idiv$-projection of $\mu$ on $\mathcal{E}(b)$ exists and is unique. 
Furthermore, if $b$ is in the relative interior of $K_{\mu}$ and $\nu$ denotes the $\Idiv$-projection of $\mu$ on $\mathcal{E}(b)$, then the Radon--Nikodym derivative $\frac{\di \nu}{\di \mu}$ is of the form $c \exp\left(\sum_{i=1}^m \kappa_i h_i \right)$ for some real constant $c > 0$ and $\kappa \in \R^m$.
\end{theorem}

In order to apply Theorem~\ref{th:idivergence} to the GMP setting, we need to describe the relative interior of $K_{\mu}$. After statement \cite[Theorem 3.3]{csiszar1975divergence}, the author mentions in some remarks that one can show that the relative interior of $K_{\mu}$ coincides with that of the convex hull of the support of the pushforward measure $h_{\#} \mu$.
We have not been able to find any formal proof of this statement in the existing literature so we give a proof of a variant of this statement below. Our proof relies on some auxiliary results from \cite{csiszar2001convex}. In this latter paper, the authors define the so-called \emph{convex core} of a finite Borel measure $\eta \in \mathcal{M}_+(\R^n)$. 
The convex core of $\eta$, denoted by $\cc(\eta)$, is the intersection of all convex Borel sets $B$ with $\eta(B) = \eta(\R^n)$. 
The authors of \cite{csiszar2001convex} show that $\cc(\eta)$ corresponds exactly to means of probability measures, which are absolutely continuous with respect to $\eta$.
\begin{theorem}[{\cite[Theorem~3]{csiszar2001convex}}]
\label{th:cc}
Given $\eta \in \mathcal{M}_+(\R^n)$, one has 
\[
\cc(\eta) = \left\{\int_{\R^n} x \di \nu : \nu \in \mathcal{P}(\R^n), \nu \ll \eta \right\}.
\]
In addition, for every $b \in \cc(\eta)$, there exists $\nu \ll \eta$ with mean $b$ such that $\frac{\di \nu}{ \di \mu}$ is bounded. 
\end{theorem}

\begin{proposition}\label{pr:cc_push}
The set $K_\mu$ coincides with the convex core of the pushforward measure $h_{\#} \mu$. In other words, we have the equality $K_{\mu} = \cc(h_{\#} \mu)$.
\end{proposition}
\begin{proof}
To prove the inclusion $K_{\mu} \subseteq \cc(h_{\#} \mu)$, fix $b \in K_{\mu}$ and let $\nu \in \mathcal{P}(X)$ be such that $\Idiv(\nu || \mu) < +\infty$ and $b = (\int_X h_i \di \nu)_i$. Since $\Idiv(\nu || \mu) < +\infty$, we have $\nu \ll \mu$. This implies that $h_{\#} \nu \ll h_{\#} \mu$. Moreover, the change of variables formula gives
\[
b_i = \int_X h_i(x) \di \nu (x) = \int_{\R^m} y_i \ \di h_{\#} \nu(y) \, .
\]
Hence, we have $b \in \cc(h_{\#} \mu)$ by Theorem~\ref{th:cc}.

To prove the opposite inclusion, fix $b \in \cc(h_{\#} \mu)$. By Theorem~\ref{th:cc}, there exists a probability measure $\eta \in \mathcal{P}(\R^m)$ such that $\eta \ll h_{\#} \mu$, $\int_{\R^m} y \di \eta(y) = b$, and such that $g = \frac{\di \eta}{ \di h_{\#} \mu}$ is bounded. Define $\nu \in \mathcal{P}(X)$ by $\nu(B) = \int_B g( h(x)) \di \mu(x)$ for all $B \in \mathcal{B}(X)$. Note that $\nu$ is indeed a probability measure, because the change of variables formula gives
\[
1 = \int_{\R^m} \di \eta = \int_{\R^m} g(y) \ \di h_{\#} \mu(y) = \int_{X} g(h(x)) \ \di \mu(x) = \int_{X} \di \nu \, .
\]
Moreover, since $\mu(B)=0$ implies $\nu(B) = 0$, one has $\nu \ll \mu$, and the corresponding Radon--Nikodym derivative $\frac{\di \nu}{\di \mu} = g \circ h$ is bounded on $X$, so $\Idiv(\nu || \mu) < +\infty$. 
In addition, for every $i$ the change of variables formula gives
\begin{align*}
b_i &= \int_{\R^m} y_i \di \eta(y) = \int_{\R^m} y_i g(y) \ \di h_{\#} \mu(y) \\
&= \int_X h_i(x) g(h(x)) \di \mu(x) = \int_X h_i(x) \nu(x) \, .
\end{align*}
Therefore, $b \in K_{\mu}$.
\end{proof}

The authors of \cite{csiszar2001convex} also define the \emph{convex support} of a finite Borel measure $\eta \in \mathcal{M}_+(\R^n)$, denoted $\cs(\eta)$, as the intersection of all closed convex sets $B$ with $\eta(B) = \eta(\R^n)$. By definition, the convex support contains the convex core, and this inclusion may be strict as the convex core is not necessarily closed. 
Nevertheless, these sets have the same relative interiors, as stated in the next proposition.

\begin{proposition}[{\cite[Lemma~1]{csiszar2001convex}}]\label{pr:cs_and_cc}
Given $\eta \in \mathcal{M}_+(\R^n)$, the sets $\cc(\eta)$ and $\cs(\eta)$ have the same relative interiors. In particular, $\cs(\eta)$ is equal to the closure of $\cc(\eta)$.
\end{proposition}

To prove the next results, we will assume that the measure $\mu$ is \emph{strictly positive}, i.e., that it satisfies $\mu(U) > 0$ for every nonempty open set $U \in \mathcal{B}(X)$. Before proceeding, we note that such a measure always exists in our setting. We refer the reader to \cite{Casteren:1994} for more information about the existence of strictly positive measures in topological spaces.

\begin{lemma}\label{le:strictly_pos}
There exists a strictly positive probability measure on $X$.
\end{lemma}
\begin{proof}
Since $X$ is a compact metric space, it is separable, i.e., there exists a sequence $x^{(1)}, x^{(2)}, \dots$ of points in $X$ such that any nonempty open set $U \in \mathcal{B}(X)$ contains at least one point of the sequence, see \cite[Theorem 8.40]{giles1987topology}. Hence, $\mu = \sum_{i = 1}^{+\infty}\frac{1}{2^{i}} \delta_{x^{(i)}}$ is a strictly positive probability measure on $X$.
\end{proof}

\begin{lemma}\label{le:conv_supp}
Suppose that $\mu \in \mathcal{P}(X)$ is strictly positive. Then, the convex support of $h_{\#} \mu$ is equal to the convex hull of $h(X)$.
\end{lemma}
\begin{proof}
Let $Z$ be the convex hull of $h(X)$. Since $X$ is compact, the set $h(X)$ is also compact, and so $Z$ is also compact, see \cite[Theorem 3.20]{rudin1991functional} or \cite[Theorem 5.35]{aliprantis2006hitchhiker}. In particular, $Z$ is a closed convex set such that $1 = h_{\#} \mu(Z) = h_{\#} \mu(\R^m)$. Therefore $\cs(h_{\#} \mu) \subseteq Z$. 

Suppose that $\cs(h_{\#} \mu) \neq Z$. By convexity of $\cs(h_{\#} \mu)$, there exists a point $y \in h(X)$ such that $y \notin \cs(h_{\#} \mu)$. Hence, by the definition of $\cs(h_{\#} \mu)$, there exists a closed convex set $B \subset \R^m$ such that $h_{\#} \mu(B) = 1$ and $y \notin B$. Since $B$ is closed, there is an open set $U\ni y$ such that $U \cap B = \emptyset$. Since $y \in h(X)$, the set $h^{-1}(U)$ is an open and nonempty set in $\mathcal{B}(X)$. Hence, $\mu(h^{-1}(U)) > 0$ and $h_{\#} \mu(U) > 0$, contradicting the assumption that $h_{\#} \mu(B) = 1$. Therefore $\cs(h_{\#} \mu) = Z$.
\end{proof}

By combining the previous results, we obtain the following theorem about the existence of a special feasible point of a GMP.

\begin{theorem}
\label{th:density_gmp}
Let $X$ be a compact nonempty subset of $\R^n$, and $\mu \in \mathcal{P}(X)$ be strictly positive.  Assume that $h_1 \equiv 1$, $b_1 > 0$ and 
that $b = (b_1,\dots,b_m)$ is in the relative interior of $K$. 
Then, the primal program \eqref{eq:gmp_primal} admits a feasible solution $\nu \in \mathcal{M}_+(X)$ with $\nu \ll \mu$, and its Radon--Nikodym derivative $\frac{\di \nu}{\di \mu}$ is of the form $c \exp\left(\sum_{i=1}^m \kappa_i h_i \right)$ for some real constant $c > 0$ and $\kappa \in \R^m$.
\end{theorem}
\begin{proof}
Without loss of generality we assume that $b_1=1$. Since $h_1 \equiv 1$, the set $h(X)$ is included in the hyperplane $\{y \in \R^m \colon y_1 = 1\}$. As $b$ is in the relative interior of $K$, it is also in the relative interior of the conic hull of $h(X)$ by Proposition~\ref{pr:tchakaloff}. Hence, there exists $\alpha > 0$ such that $\alpha b$ is in the relative interior of the convex hull of $h(X)$. Since $b_1 = 1$ and the convex hull of $h(X)$ is included in $\{y \in \R^m \colon y_1 = 1\}$, we have that $\alpha = 1$. Hence, $b$ is in the relative interior of the convex hull of $h(X)$. By Lemma~\ref{le:conv_supp} this implies that $b$ is in the relative interior of the convex support of $h_{\#} \mu$, which by Proposition~\ref{pr:cs_and_cc} implies that $b$ is in the relative interior of the convex core of $h_{\#} \mu$. This in turn implies, by Proposition~\ref{pr:cc_push}, that $b$ is in the relative interior of $K_{\mu}$. The claim then follows from Theorem~\ref{th:idivergence}, namely, from the fact that $\nu\in \mathcal{E}(b)$.
\end{proof}

\begin{corollary}\label{cor:feasible_measure}
Let $X$ be a compact nonempty subset of $\R^n$. Assume that $h_1 \equiv 1$, $b_1 > 0$ and that $b = (b_1,\dots,b_m)$ is in the relative interior of $K$. Then, the primal program \eqref{eq:gmp_primal} admits a feasible solution $\nu \in \mathcal{M}_+(X)$ such that for any nonnegative continuous function $g \colon X \to \R_{\ge 0}$ we have 
\[
\int_X g \,\di \nu = 0 \iff g(x) = 0 \text{ for all $x \in X$} \, .
\]
\end{corollary}
\begin{proof}
Let $\mu$ be any strictly positive probability measure on $X$ (such a measure exists by Lemma~\ref{le:strictly_pos}). Let $\nu$ be the measure constructed from $\mu$ in Theorem~\ref{th:density_gmp} and let $g \colon X \to \R_{\ge 0}$ be any nonnegative continuous function. If $g(x) = 0$ for all $x \in X$, then $\int_X g(x) \di \nu(x) = 0$. Conversely, suppose that $g(x) \neq 0$ for some $x \in X$. Then, there exist $\varepsilon > 0$ and an open set $U \in \mathcal{B}(X)$ such that $x \in U$ and $cg(y)\exp(\sum_{i = 1}^m \kappa_i h_i(y)) > \varepsilon$ for all $y \in U$. In particular, $\int_X g \di \nu = \int_X cg\exp(\sum_{i = 1}^m \kappa_i h_i) \di \mu \ge \int_U cg\exp(\sum_{i = 1}^m \kappa_i h_i) \di \mu \ge \varepsilon \mu(U) > 0$.
\end{proof}

\subsection{Strong duality of semidefinite programs}\label{ssec:GMP_relax_density_proof}

In this subsection, we present an alternative proof of Theorem \ref{th:attainSOS}, which comes from an application of the strong duality of semidefinite programs and uses the strictly positive measure from the previous subsection. We repeat the statement for convenience. 

\begin{theorem}[Repetition of Theorem \ref{th:attainSOS}]
\label{th:attainSOS2}
Let $X$ be a compact subset of $\R^n$ satisfying Assumption \ref{hyp:ball}. 
Suppose that either
\begin{enumerate}[(i)]
    \item The ideal $\Ip=(p_1,\dots,p_s)$ is real radical, and that $\p=\{p_1,\dots,p_s\}$ forms a Gröbner basis of $\Ip$ with respect to the graded lexicographic order, or
    \item $\p=\{0\}$ and $X$ has nonempty interior. 
\end{enumerate}
Assume further, that $h_1\equiv1$, $b_1 > 0$, and that $b = (b_1,\dots,b_m)$ is in the relative interior of $K$. 
Then the semidefinite program \eqref{eq:dual_relax} attains its optimal value for all $t \geq t_{\min}$.
\end{theorem}
\begin{proof}
\textit{(ii)}  By Theorem \ref{th:density_gmp}, if $X$ has nonempty interior, we have a measure in the relative interior of cone of feasible solutions, i.e., a strictly feasible point. Taking the moments of this measure, up to degree $2t$, gives a strictly feasible solution to the semidefinite program \eqref{eq:gmp_relax}, which is the dual of \eqref{eq:dual_relax}. To see this, let $\mom\in \R^{\N_{2t}^n}$ be the truncated moment sequence of the measure $\nu$ from Corollary~\ref{cor:feasible_measure}. We then have that  $p^{T}M_{t-t_{g_j}}(g_j \mom)p = \int_X g_jp^2 \di \nu=0$, if and only if $p=0$, by the fact that $X$ has nonempty interior and $g_j\geq 0$ on $X$. So the matrices are positive definite and $\mom$ is strictly feasible. For details on manipulation of moment matrices see \cite[Lemma 4.2]{laurent_survey_2009} and its proof.

Thus, by the strong duality of semidefinite programs the supremum in \eqref{eq:dual_relax} is attained when the value of the primal program \eqref{eq:gmp_relax} is lower bounded. 
This latter lower bound statement has been already shown in the proof of Theorem~\ref{th:attainSOS1}. 

\textit{(i)} Since the nonempty interior assumption is not satisfied in this case, the argument in the proof of item \textit{(ii)} does not immediately apply. To use the strong duality of semidefinite programs, we need that a solution of the primal \eqref{eq:gmp_relax} has a strictly feasible solution. However, this condition does not hold for \eqref{eq:gmp_relax}, when $\p\neq\{0\}$. To remedy this, we use a similar proof technique to that of \cite[Lemma 1]{henrion2012innerapprox}. Namely, we formulate a reduced semidefinite program using the equality constraints defining $X$, which upper bounds \eqref{eq:gmp_relax}. We then show this reduced problem has a strictly feasible solution, which allows us to apply strong duality. This is the content of Appendix \ref{ap:algred}, which then allows us to complete the proof in Appendix \ref{ap:pfofthm}.
\end{proof}

\section{Applications of main results}\label{sec:apps}

In this section, we explore applications of Theorem \ref{th:gmp_int} and Theorem \ref{th:attainSOS}.

\subsection{Product of spheres}\label{ssec:sphs}

We show an application of Theorem \ref{th:attainSOS} to a specific set with empty interior, the product of spheres. Two technical lemmas are needed to satisfy the algebraic assumptions of Theorem \ref{th:attainSOS} in this case. We present them in Appendix \ref{ap:spheres_alg}.

We fix some notation. Consider the tuple $(n_1,\dots,n_s)\in\N^s$, and set $N= \sum_{l =1}^s n_l$. We assume that $n_l\geq 2$ for all $l=1,\dots,s$. Set $x_l = (x_{l1},\dots x_{ln_l})$, for each $l=1,\dots,s$, and $x = (x_1 , \dots , x_s) \in  \R^N$. We denote the polynomial ring over these variables with the shorthand $\R[x]$. The monomials of this ring are indexed by $\N^{n_1+\dots + n_s}$, which we denote by $\mathcal{N}$. For the truncation to degree $t$, write $\mathcal{N}_t$. We set $\mathbb{S}^{n_l-1} = \{x_l\in \R^{n_l}\mid 1 - \|x_l\|^2 = 0\}$ as the unit sphere in $\R^{n_l}$, $\varphi_l = 1 - \|x_l\|^2$, and $\bm{\varphi}_s = (\varphi_1,\dots,\varphi_s)$. We refer to the following set as the \emph{product of spheres},
\begin{equation}\label{eq:prodofsphs}
X = V_{\R}(\varphi_1,\dots,\varphi_s) = \mathbb{S}^{n_1-1}\times \dots \times \mathbb{S}^{n_s-1} \subset \R^{n_1}\times \dots \times \R^{n_s}=\R^N.
\end{equation}

Note that one does not need the redundant ball constraint from Assumption \ref{hyp:ball} (i). 
Similarly to the proof of Theorem~\ref{th:attainSOS1}, the value $\primal_t$ of the primal program \eqref{eq:gmp_relax} is lower bounded by leveraging the multiple sphere constraints together with $h_1 \equiv 1$ and $b_1 >0$. 
In this case all entries of any feasible $z$ for \eqref{eq:gmp_relax} are uniformly bounded in absolute value by \cite[Lemma 5]{tacchi2022convergence}. 

\begin{corollary}\label{co:attainSOS_sphs}
Let $X$ be the product of spheres, as in \eqref{eq:prodofsphs}. Assume that $h_1 \equiv 1$, $b_1 > 0$ and that $b = (b_1,\dots,b_m)$ is in the relative interior of $K$. Then the semidefinite program \eqref{eq:dual_relax}, with $\p= \bm{\varphi}_s$, attains its optimal value, for all $t \geq t_{\min}$.
\end{corollary}
\begin{proof}
By Lemma \ref{le:realrad} and Lemma \ref{le:grobner_basis}, the real radical and Gröbner basis assumptions of Theorem \ref{th:attainSOS} are satisfied, and the result follows.
\end{proof}

\subsection{Background for application problems}\label{ssec:apps_background}

To describe the problems to which we will apply our results, we briefly introduce some notions from tensor theory and quantum information theory. References for tensor theory include \cite{hogben_linalghandbook_2014,comon_symtensors_2008,landsberg_tensors_2012}, and references for the relevant problems from quantum information theory include \cite{nielsen_qit_2000,watrous_qit_2018}.

\subsubsection{Tensors}

Let $a \in \N$. 
Given $n_1,\dots, n_a \in \N$, we define a real tensor $\mathcal{A}$ of order $a$, as an element of the space $\R^{n_1\times \dots \times n_a}$. Suppose $n_1=\dots= n_a =n\in \N$, then $\mathcal{A}$ is called symmetric, if it is invariant by permutation of its indices. We say a tensor $\mathcal{Q} \in \R^{n_1\times \dots \times n_a}$ is rank one, if it can be written as the outer product of unit vectors $u_i \in \R^{n_i}$ times a nonzero scalar $q$, i.e.,
\begin{equation}\label{rankonetensor}
	\mathcal{Q} = q u_1 \otimes \dots \otimes u_a, \; q\in \R \setminus \{0\}, \; \|u_i\| = 1 \;\text{for all}\; i =1,\dots, a.
\end{equation}
For a symmetric rank one tensor $\mathcal{Q}$, one can write
\[
\mathcal{Q} = q u \otimes \dots \otimes u = q u^{\otimes a}, \; \text{where} \; q \in \R\setminus\{0\}, \; \|u\| = 1.
\]
For two tensors $\mathcal{A}, \mathcal{A}' \in \R^{n_1\times \dots \times n_a}$, the Hilbert Schmidt inner product is given by
\[\langle \mathcal{A},\mathcal{A}'\rangle = \sum_{j_1=1}^{n_1} \; \dots \sum_{j_a=1}^{n_a} \; \mathcal{A}_{j_1\dots j_a}\; \mathcal{A}'_{j_1\dots j_a}.\]
We can associate a polynomial to a tensor $\mathcal{A} \in \R^{n_1\times \dots \times n_a}$ through the inner product, 
\begin{equation}\label{eq:tensor_poly}
    \mathcal{A}(x) = \;\langle \mathcal{A}, x_1 \otimes \dots \otimes x_a \rangle \; = \;\sum_{j_1=1}^{n_1} \dots \sum_{j_a =1}^{n_a}\; \mathcal{A}_{j_1\dots j_a} \; x_{1 j_1} \dots x_{a j_a}.
\end{equation}
Note that, the polynomial above is multilinear in each vector $x_i=(x_{i 1},\dots,x_{i n_i})$. 

In the symmetric case, $\mathcal{A}(x)$ is a homogeneous polynomial of degree $a$ in $n$ variables,
\begin{equation}\label{eq:tensor_poly_sym}
    \mathcal{A}(x) = \;\langle \mathcal{A}, x^{\otimes a} \rangle \; = \;\sum_{j_1,\dots,j_a=1}^{n} \mathcal{A}_{j_1\dots j_a} \; x_{j_1} \dots x_{j_a}.
\end{equation}
By dehomogenizing, we also have a one to one association between order $a$ symmetric tensors in $n$ dimensions and polynomials in of degree less than or equal to $a$ in $n-1$ variables. In other words, we set $x_n=1$ in \eqref{eq:tensor_poly_sym}, and write
\begin{equation}\label{eq:tensor_poly_sym_dehomo}
\Tilde{\mathcal{A}}(x) = \mathcal{A}(x_1,\dots,x_{n-1},1).
\end{equation}

\subsubsection{Quantum information}

Given $n \in \N$, let $\mathcal{H}^n = \{A \in \C^{n\times n} \; | \; A=A^*\}$ be the space of $n\times n$ Hermitian matrices, where $A^*$ is the the conjugate transpose of $A$.  We denote by $\D(\C^n) \subset \mathcal{H}^n$ the set of \emph{normalized quantum states}, i.e., Hermitian matrices with nonnegative eigenvalues and trace equal to one.

A \emph{quantum bipartite state} $\rho$ is an element of the space $\mathcal{H}^n \otimes \mathcal{H}^n$. The cone of \emph{separable bipartite states} is defined as
\begin{equation}\label{eq:SEP_def}
\begin{aligned}
\sep = \cone\Big\{\rho\in \mathcal{H}^n \otimes \mathcal{H}^n \; | \;
\rho = \sum_{\ell=1}^r w_{\ell} \qu^{(\ell)} {\qu^{(\ell)}}^{*}  \otimes \qv^{(\ell)} {\qv^{(\ell)}}^{*} &,\\
w_{\ell} \geq 0, \qu^{(\ell)},\qv^{(\ell)}\in \C^n,\; r\in\N &\Big\}
\end{aligned}
\end{equation}
For a separable bipartite state $\rho$, the smallest $r$, such that the equality in \eqref{eq:SEP_def} holds, is called the \emph{separable rank} of $\rho$.

\subsubsection{GMP with complex variables}\label{ssec:complex_vars}

To formulate problems arising from quantum information theory, we introduce formalism for GMPs over the complex numbers.

A polynomial $f\in \C[x,\overline{x}]$ has the form $f(x,\overline{x}) = \sum_{\alpha,\beta} f_{\alpha,\beta}x^{\alpha}\overline{x}^{\beta}$, i.e., $f$ is a polynomial in $2n$ variables $x_1,\dots,x_n,\overline{x}_1,\dots,\overline{x}_n$. Such a polynomial $f$ defines a function
\[
\begin{aligned}
\C^n &\longrightarrow \C \\
c &\longmapsto f(c,\overline{c}),
\end{aligned}
\]
which we denote by $f(x)$. We say a polynomial $f$ is \emph{Hermitian}, if it is real valued, i.e., $f(c) = \Real(f(c,\overline{c}))$ for all $c\in \C^n$. This allows us to formulate a GMP instance with complex variables,
\begin{equation}\label{eq:gmp_primal_Csphs}
\begin{aligned}
    &\inf_{\mu} \;\; \int_X f \di \mu \\
    &\st \; \int_X h_i \di \mu  = b_i, \ i=1,\dots,m, \\
    & \qquad \int_X \overline{h}_i \di \mu  = \overline{b}_i, \ i=1,\dots,m, \\
    & \qquad \mu \in \mathcal{M}_+(X),
\end{aligned}
\end{equation}
with $f,h_1,\dots,h_m\in\C[x,\overline{x},y,\overline{y}]$, $f$ Hermitian, $b_1,\dots,b_m \in \C$, and $X\subset\C^{2n}$. 
As before, we have the dual problem,
\begin{equation}\label{eq:gmp_dual_Csphs}
\begin{aligned}
    & \sup_{\lambda} \;\; \sum_{i=1}^m \Big(b_i \lambda_i + \overline{b}_i \overline{\lambda}_i \Big) \\
    & \st \; f - \sum_{i=1}^m\Big( h_i\lambda_i + \overline{h}_i\overline{\lambda}_i \Big)\geq 0 \;\;\text{on}\; X, \\
    & \qquad \lambda\in \C^m.
\end{aligned}
\end{equation}
The problems \eqref{eq:gmp_primal_Csphs} and \eqref{eq:gmp_dual_Csphs} are over the complex numbers. However, using the procedure in \cite[Appendix A.1]{gribling2022momdps}, they can be changed, piece by piece, to be over the real numbers. 
Let $\ib = \sqrt{-1} \in \C$. 
Polynomials over $\C[x, \overline{x}, y, \overline{y}]$ can be changed into polynomials over $\R[x_{\Real},x_{\Imaginary},y_{\Real},y_{\Imaginary}]$ by applying the change of variables 
\begin{equation}\label{eq:change_vars}
x = x_{\Real} + \ib x_{\Imaginary} \qquad y = y_{\Real} + \ib y_{\Imaginary}. 
\end{equation}
This can be readily extended to vectors and matrices by applying the corresponding operations entrywise. 
Thus, we write 
\begin{equation*}
\begin{aligned}
h(x,\overline{x},y,\overline{y}) &= \Real(h(x,\overline{x},y,\overline{y})) + \ib \Imaginary(h(x,\overline{x},y,\overline{y})) \\
&= h_{\Real}(x_{\Real},x_{\Imaginary},y_{\Real},y_{\Imaginary}) + \ib h_{\Imaginary}(x_{\Real},x_{\Imaginary},y_{\Real},y_{\Imaginary}),
\end{aligned}
\end{equation*}
where $h_{\Real},h_{\Imaginary}\in\R[x_{\Real},x_{\Imaginary},y_{\Real},y_{\Imaginary}]$. Note that for $f$ Hermitian, we have $f=\Real(f)$, and so $f$ and $f_{\Real}$ define the same functions on $\C^n$, but live in different polynomial rings. 

Finally, note that the change of variables in \eqref{eq:change_vars} describes a bijection between $\C^{2n}$ and $\R^{4n}$, where 
\[(x,y) = (x_{\Real}+\ib x_{\Imaginary},y_{\Real}+\ib y_{\Imaginary}) \; \longmapsto \; (x_{\Real},x_{\Imaginary},y_{\Real},y_{\Imaginary}).\]
Set $X_{\Real}\subset\R^{4n}$ as the image of $X$ under this bijection. For a measure $\mu$ on $X$, we set $\mu_{\Real}$ to be equal to the pushforward of $\mu$ by this bijection. This allows us to write the problems \eqref{eq:gmp_primal_Csphs} and \eqref{eq:gmp_dual_Csphs} over real variables, as follows
\begin{equation}\label{eq:gmp_primal_CtoRsphs}
\begin{aligned}
    &\inf_{\mu_{\Real}} \;\; \int_X f_{\Real} \, \di \mu_{\Real} \\
    &\st \; \int_X h_{i,\Real} \,\di \mu_{\Real}  = \Real(b_i), \ i=1,\dots,m \\
    & \qquad \int_X h_{i,\Imaginary} \, \di \mu_{\Real}  = \Imaginary(b_i), \ i=1,\dots,m \\
    & \qquad \mu_{\Real} \in \mathcal{M}_+(X_{\Real}), 
\end{aligned}
\end{equation}
with the dual problem,
\begin{equation}\label{eq:gmp_dual_CtoRsphs}
\begin{aligned}
    & \sup_{\lambda_{\Real}, \lambda_{\Imaginary}} \;\; \sum_{i=1}^m \Big(\Real(b_i)\lambda_{i,\Real} + \Imaginary(b_i)\lambda_{i,\Imaginary} \Big)\\
    & \st \; f - \sum_{i=1}^m \Big( h_{i,\Real}\lambda_{i,\Real} + h_{i,\Imaginary}\lambda_{i,\Imaginary}\Big) \geq 0\;\; \text{on} \; X_{\Real}, \\
    & \qquad \lambda_{\Real}, \lambda_{\Imaginary}\in\R^m.
\end{aligned}
\end{equation}

\subsection{Hierarchies over products of spheres}\label{ssec:apps_sphs}

We introduce three problems, their associated GMPs, and the applications of Theorem \ref{th:gmp_int} and Corollary \ref{co:attainSOS_sphs} to them.

Entanglement is a key property in quantum information theory, thus detecting whether a given state is entangled or separable is an important question. The Doherty Parrilo Spedalieri (DPS) hierarchy, first introduced in \cite{doherty2002dps}, is a hierarchy of outer approximations of the cone of separable bipartite states. See \cite[Section 4]{fang2021sphere} for a concise introduction. In quantum information literature, the formulation of the DPS hierarchy is usually given in tensor notation over the complex numbers. In \cite{gribling2022momdps,dressler2022hermitian,li2020separability}, the following GMP formulation of the DPS hierarchy is studied:
\begin{equation}\label{eq:dps}
\begin{aligned}
&\inf_{\mu}  \; \; \int_X 1 \di\mu  \qquad\qquad\qquad\qquad &&\sup_{\Lambda}  \; \; \Tr (\rho \Lambda) \\
&\st \;  \int_X xx^* \otimes yy^* \di\mu =\rho &&\st \; 1 - (x \otimes y)^* \Lambda (x \otimes y) \geq 0 \text{ on } X \\
&\qquad \mu \in\mathcal{M}_+(X) && \qquad \Lambda \in \mathcal{H}^n,
\end{aligned}
\end{equation}
where $X$ is the product of two complex spheres,
\begin{equation}\label{eq:complex_sphs}
X=\mathbb{S}^{n-1}_{\C}\times\mathbb{S}^{n-1}_{\C},\;\; \text{with} \;\mathbb{S}^{n-1}_{\C} = \{x\in \C^n \; | \; 1-\|x\|^2 = 0\}.
\end{equation}
The real analog of $X$ is the product of two spheres, denoted by
\[
\begin{aligned}
X_{\Real} &= \mathbb{S}^{2n-1} \times \mathbb{S}^{2n-1}\\
&= \Big\{(x_{\Real},x_{\Imaginary},y_{\Real},y_{\Imaginary})\in \R^{4n}\; | &&1-\|x_{\Real}\|^2 -\|x_{\Imaginary}\|^2 =0\\
&\; &&1-\|y_{\Real}\|^2 -\|y_{\Imaginary}\|^2 =0\Big\}.
\end{aligned}
\]
In the notation of \eqref{eq:gmp_primal}, we have $h_{ij,kl}=x_i \overline{x}_j y_k \overline{y}_l$ and  $b_{ij,kl}=\rho_{ij,kl}$, where $i,j,k,l\in \{1,\dots,n\}$. See \cite[Section 5.1]{britz2025CLDUI}, for a summary of the construction of this GMP formulation. 

The second problem we consider comes from quantum optimal transport theory.  
The classical Wasserstein distance essentially measures the minimal cost required to transform one probability distribution into another. 
To adapt this concept to the quantum setting, various definitions of quantum Wasserstein distances have been introduced; see, for example, the works of \cite{nielsen06} and \cite{de2021quantum}. 
In a recent study \cite{chhatoi2025wasserstein}, the authors consider the order 2 quantum Wasserstein distance proposed by \cite{beatty2025order}, and specifically formulate it as a GMP instance. Given two normalized quantum states $\tau, \omega \in \D(\C^n)$, a \textit{quantum transport plan} between $\tau$ and $\omega$ is a finite set of triples $\{(w_{\ell}, \qu_{\ell}, \qv_{\ell})\}_{\ell}$ such that 
\begin{align}
\sum_{\ell} w_{\ell} \, \qu^{(\ell)} {\qu^{(\ell)}}^* = \tau \quad \text{ and} \quad\sum_{\ell} w_{\ell} \,  \qv^{(\ell)} {\qv^{(\ell)}}^* = \omega, 
\end{align}
where $w_{\ell} > 0$, $\sum_{\ell} w_{\ell} =1$, $\qu^{(\ell)}, \qv^{(\ell)} \in \C^n$,  $\|\qu^{(\ell)}\|=\| \qv^{(\ell)}\| = 1$. 
Let $\quantum(\tau,\omega)$ be the set of all such quantum transport plans between two states $\tau$ and $\omega$. 
Given a plan $\xi=\{(w_{\ell}, \qu^{(\ell)}, \qv^{(\ell)})\}_{\ell}$, the order $2$ \textit{quantum transport cost} is defined as 
\begin{align}
T_2(\xi) := \sum_{\ell} w_{\ell} \Tr \left( (\qu^{(\ell)} {\qu^{(\ell)}}^* - \qv^{(\ell)} {\qv^{(\ell)}}^*) \overline{(\qu^{(\ell)} {\qu^{(\ell)}}^* - \qv^{(\ell)} {\qv^{(\ell)}}^*)} \right),
\end{align}
and the order $2$ quantum Wasserstein distance is given by
\begin{align}
\label{eq:qwtwo}
W_2(\tau,\omega) := \Big(\inf_{\xi \in \quantum(\tau,\omega)} T_2(\xi)\Big)^{1/2}.
\end{align}
In \cite[Section 3]{chhatoi2025wasserstein}, the precise moment formulation of the Wasserstein distance, and its conversion to real variables, are presented. 
It is shown in \cite[Theorem 8]{chhatoi2025wasserstein} that $W_2^2(\tau,\omega)$ is the optimal value of the following GMP instance:
\begin{equation}
\label{eq:qwgmp}
\begin{aligned}
W^2_2(\tau,\omega) = \inf_{\mu} \quad  & \int_{X} f \di  \mu   \\	
\st 
\quad & \int_{X} x x^* \di  \mu = \tau, \\
\quad & \int_{X} y y^* \di  \mu = \omega, \\
\quad & \mu \in \mathcal{M}_+(X),
\end{aligned}
\end{equation}
where $f(x,\overline{x},y,\overline{y})= \Tr\left((x x^* - y y^*) \overline{(x x^* - y y^*)}\right)$, and $X$ is again the product of two complex spheres, as in \eqref{eq:complex_sphs}. In the notation of \eqref{eq:gmp_primal}, we have $h_{ij}=x_i \overline{x}_j$ with $b_{ij}=\tau_{ij}$, and $h_{kl}=y_k \overline{y}_l$ with $b_{kl}=\omega_{kl}$, where $i,j,k,l\in \{1,\dots,n\}$.
The dual of \eqref{eq:qwgmp} is 
\begin{equation}
\label{eq:qwdual}
\begin{aligned}
\sup_{\Lambda,\Gamma} \quad  & \Tr (\tau \Lambda + \omega \Gamma) \\	
\text{s.t.}
\quad & f - x^* \Lambda x -  y^* \Gamma y \geq 0 \text{ on } X, \\
\quad & \Lambda,\Gamma \in \mathcal{H}^n.
\end{aligned}
\end{equation}

For the final application, we consider the problem of finding the best rank one approximation of a given tensor $\mathcal{A} \in \R^{n_1\times \dots \times n_a}$. Setting $\mathcal{X} = q \,x_1 \otimes \dots \otimes x_a$ as an indeterminate rank one tensor, the problem can be formulated as,
\begin{equation}\label{eq:bestrankone}
\begin{aligned}
&\min_{\mathcal{X}} \; \|\mathcal{A} - \mathcal{X}\|^2 \\
&\st\;\; \rank(\mathcal{X})= 1 \\
& \qquad \; \mathcal{X} \in \R^{n_1\times \dots \times n_a}.
\end{aligned}
\end{equation}
In \cite{nie_rankone_2014}, this is transformed into the following GMP instance and its dual polynomial optimization problem,
\begin{equation}\label{eq:bestrankone_gmp}
\begin{aligned}
&\inf_{\mu}  \; \; \int_X \mathcal{A}(x) \di\mu  \qquad\qquad &&\sup_{\lambda} \; \; \lambda \\
&\st \;  \int 1 \di\mu = 1 &&\st \; \mathcal{A}(x) - \lambda \geq 0 \text{ on } X \\
&\qquad \mu \in\mathcal{M}_+(X) && \qquad \lambda \in \R,
\end{aligned}
\end{equation}
where $X$ is the product of real spheres, as in \eqref{eq:prodofsphs}. In the notation of \eqref{eq:gmp_primal}, we have $h\equiv1$ with $b=1$. Note that the primal problem in \eqref{eq:bestrankone_gmp} is referred to in the literature as a moment problem, i.e., not `generalized'. This does not interfere with the application of our results to it.

Recall that GMPs over the product of complex spheres can be changed to be over real variables, and hence over the product of real spheres, as described in Subsection \ref{ssec:complex_vars}. 
Thus, by Theorem \ref{th:gmp_int} and Corollary \ref{co:attainSOS_sphs}, under the assumption of $b$ being in the relative interior of $K$, the GMP formulations of the three problems presented above, and the relaxations thereof, have duality attainment.
\begin{corollary}\label{cor:sphs_app}
Under the relative interior assumption in Corollary \ref{co:attainSOS_sphs}, the 3 problems presented in this subsection have duality attainment. More specifically:
\begin{itemize}
    \item Suppose the bipartite state $\rho$ is in the relative interior of the cone of separable states, i.e., $\rho\in\relint(\sep)$. Then the GMP formulation of the DPS hierarchy, as in \eqref{eq:dps}, and the relaxations thereof, have duality attainment.
    \item Suppose the states $\tau$ and $\omega$ are in the relative interior of the cone of normalized states $\mathcal{D}(\C^n)$, i.e., $\tau$ and $\omega$ are positive definite. Then the GMP formulation of the order 2 quantum Wasserstein distance, as in \eqref{eq:qwgmp}, and the relaxations thereof, have duality attainment.
    \item Given a tensor $\mathcal{A} \in \R^{n_1\times \dots \times n_a}$, the GMP formulation to find its best rank one approximation, as in \eqref{eq:bestrankone_gmp}, and the relaxations thereof, have duality attainment.
\end{itemize}
\end{corollary}
Note that \cite[Theorem 9]{chhatoi2025wasserstein} shows duality attainment of \eqref{eq:qwgmp}, but not of its relaxations.

\subsection{Hierarchies over sets with nonempty interior}\label{ssec:apps_nonempty}

We briefly introduce the problem of positive symmetric tensor decomposition and an associated GMP instance. We then show the application of Theorem \ref{th:gmp_int} and Theorem \ref{th:attainSOS}.

In \cite{gamertsfelder2025effective}, the problem of positive symmetric tensor decomposition is considered. Finding the positive decomposition of an order $a$ symmetric tensor $\mathcal{A}\in \R^{n+1 \times \dots \times n+1}$ consists of finding the $v_1,\dots,v_{\ell}\in \R^{n+1}$ and $w_1,\dots,w_{\ell} > 0$ with
\[
\mathcal{A}= \sum_{i=1}^{\ell} w_i v_i^{\otimes a},
\]
such that $\ell$ is minimal. Note that, if all the diagonal entries of $\mathcal{A}$ are zero and $\mathcal{A}$ is not zero, such a decomposition cannot exist. Hence, we assume that at least one nondiagonal entry is nonzero, and, without loss of generality, we suppose that the entry $\mathcal{A}_{n+1 \dots n+1}\neq0$.

In \cite[Section 3.1]{gamertsfelder2025effective}, this problem is handled by solving a GMP instance over the real ball, $X=\mathbb{B}^n = \{x\in \R^n \;|\; 1-\|x\|^2 \geq 0\}$. While the GMP instance is not equivalent to the tensor decomposition problem, the solution of the GMP can be used to retrieve a positive symmetric tensor decomposition. The connection between the tensor problem and a GMP instance is made by using the unique correspondence between the order $a$ symmetric tensor $\mathcal{A}\in \R^{n+1 \times \dots \times n+1}$ and the polynomial $\Tilde{\mathcal{A}}(x)$, as in \eqref{eq:tensor_poly_sym_dehomo}. Note that $\Tilde{\mathcal{A}}(x)$ is a polynomial in $n$ variables and is of degree at most $a$. With $\Psi$ being a generic SOS polynomial, the GMP and its dual are formulated as,
\begin{equation}\label{eq:tensor_decomp}
\begin{aligned}
&\inf_{\mu}  \; \; \int_X \Psi\, \di\mu  \qquad\qquad\qquad\qquad &&\sup_{\lambda}  \; \; \langle \Tilde{\mathcal{A}}(x),\sum_{\alpha\in \N^n_{a}}  x^{\alpha}\lambda_\alpha\rangle_a \\
&\st \;  \int_X x^{\alpha}\, \di\mu = \langle \Tilde{\mathcal{A}}(x),x^{\alpha}\rangle_a &&\st \; \Psi - \sum_{\alpha\in \N^n_{a}}  x^{\alpha}\lambda_\alpha \geq 0 \text{ on } X \\
&\qquad \mu \in\mathcal{M}_+(X) && \qquad \lambda \in \R^{\N^n_a},
\end{aligned}
\end{equation}
where the apolar product is defined as,
\[
\langle f,g \rangle_a = \sum_{\alpha\in \N^n_a} \begin{pmatrix}
a \\
\alpha
\end{pmatrix}^{-1} f_{\alpha} g_{\alpha},\;\;
\text{with} \;\; \begin{pmatrix}a \\ \alpha \end{pmatrix} = \frac{a!}{(a-|\alpha|)! \alpha_1! \dots \alpha_n!}.
\]
In the notation of \eqref{eq:gmp_primal}, we have $h_{\alpha}=x^{\alpha}$ with $b_{\alpha}= \langle \Tilde{\mathcal{A}}(x),x^{\alpha}\rangle_a$, where $\alpha\in \N^n_a$. Note that, we have $b_0=\langle \Tilde{\mathcal{A}}(x),x^{(0,\dots, 0)}\rangle_a = \mathcal{A}_{n+1 \dots n+1}$, which can be scaled to be equal to one.

To apply Theorem \ref{th:gmp_int} and Theorem \ref{th:attainSOS}, note that in this case
\[
K = \left\{ \left(\int_X x^{\alpha} \di \mu\right)_{\alpha\in\N^n_a} :  \mu \in \mathcal{M}_+(X) \right\},
\]
which is normally referred to as the truncated moment cone.
\begin{corollary}
Given an order $a$ symmetric tensor $\mathcal{A}\in\R^{n+1\times\dots\times n+1}$, and its associated polynomial $\Tilde{\mathcal{A}}(x)$, as in \eqref{eq:tensor_poly_sym_dehomo}, define the sequence \[\mom_{\mathcal{A}} = \Big(\langle \Tilde{\mathcal{A}}(x),x^{\alpha}\rangle_a \Big)_{\alpha\in \N^n_a}.
\]
Suppose $\mom_{\mathcal{A}}$ is in the relative interior of the truncated moment cone $K$. Then the GMP associated to the positive symmetric tensor decomposition problem for $\mathcal{A}$, as in \eqref{eq:tensor_decomp}, and the relaxations thereof, have duality attainment.
\end{corollary}

Note that \cite[Theorem 3.6]{gamertsfelder2025effective} shows the attainment of the infinite dimensional dual GMP. 
This is also shown in \cite[Proposition 3.6]{nie_atruncKprob_2014}, under the assumption of `strict $X$ positivity' on the Riesz functional $L_{\mom_{\mathcal{A}}}$, meaning
\[
L_{\mom_{\mathcal{A}}}(p) > 0 \;\; \text{for all} \;\; p\in \R[x]_a \setminus \{0\} \;\; \st \;\; p \geq 0 \;\; \text{on} \;\; X.
\]
As mentioned in Remark \ref{rk:gmp_int}, this coincides with our condition regarding membership in the relative interior of $K$.

\section{Conclusion and perspectives}\label{sec:conclusion}

We studied duality attainment for the Generalized Moment Problem (GMP) with polynomial input data and measures supported on basic compact semialgebraic sets. Under a relative interior condition on the prescribed moment vector, we proved attainment of the infinite dimensional GMP dual. We also established dual attainment for every level of the associated moment–SOS hierarchy. 
For the latter, we leveraged two distinct proof strategies. 
Our first strategy relies on establishing that the truncated quadratic module involved in each dual SOS program is closed. 
Our second strategy relies on entropy-based constructions and $\Idiv$-projection techniques to build feasible exponential densities, which yield strictly feasible points for the primal pseudo-moment relaxations. 
The results apply both when the support set has nonempty interior and when it is contained in the vanishing set of prescribed polynomials forming a Gröbner basis of the ideal they generate, which is assumed to be real radical. 
They can be readily applied to various GMP instances arising from tensor approximation and quantum information. 

Our framework is currently restricted to GMP instances involving finitely many linear equality constraints and measures supported on sets described by polynomial inequalities. 
Several directions deserve further investigation. 

A first extension concerns GMPs involving multiple measures, often used in the context of semialgebraic set characterization, and analysis and control of nonlinear systems with polynomial dynamics. 
In this case, the underlying optimization problems often involve infinitely (countably) many linear equality constraints. 
When considering such instances, the existence of strictly feasible solutions at every relaxation level does not necessarily imply that the infinite dimensional GMP dual is attained, see, e.g., \cite{henrion2009approximate}. 
One subclass of interest would be GMP instances whose input data exhibit correlative sparsity; see, e.g., \cite[Chapter 4]{lasserre_momentsandapps_2010} or \cite[Chapter 3]{magron2023sparse}. 
For such instances one should be able to develop dual attainment guarantees compatible with sparse moment–SOS hierarchies. 

Eventually, one could naturally extend our framework to GMP instances involving measures supported on sets described by polynomial matrix inequalities. 
One such instance is presented in \cite[Section 3]{gribling2022momdps} for the problem of finding the best separable rank of a given separable bipartite state. A sequence of semidefinite relaxations of this GMP give a hierarchy of lower bounds on the best separable rank. This GMP, and the relaxations thereof, are over the product of two complex balls.
After a translation of the GMP, stated in \cite{gribling2022momdps}, to the form of \eqref{eq:gmp_primal}, one could likely apply our results  to show duality attainment. 

\appendix

\section{Algebraic results}\label{ap:alg_results}

The following appendices contain technical algebraic results and constructions used to complete the proof Subsection \ref{ssec:GMP_relax_density_proof}. In proving these results, a detailed formulation of an algebraic reduction of \eqref{eq:gmp_relax} and \eqref{eq:dual_relax} was needed. Such a formulation, in the case of the real vanishing set $V_{\R}(\Ip)$ being finite, is constructed in \cite[Section 2]{laurent2007varieties}. It is also mentioned, at a more abstract level, for the general case in \cite[Sections 3 and 4]{marshall_optim_2003}. It is also discussed in \cite[Section 6.2]{laurent_survey_2009}. To our knowledge, a detailed formulation of an algebraic reduction for a general vanishing set is not available in the literature. Thus, we deal with the technical details of such a formulation in Appendix \ref{ap:algred}.

Appendix \ref{ap:pfofthm} is then dedicated to finishing the proof of Theorem \ref{th:attainSOS2} started in Subsection \ref{ssec:GMP_relax_density_proof}, using the above mentioned algebraic reduction. Finally, Appendix \ref{ap:spheres_alg} contains technical lemmas used in the proof of Corollary \ref{co:attainSOS_sphs}, which deals with the special case of the product of spheres.

\subsection{Algebraic reduction of the GMP relaxations}\label{ap:algred}

We consider the case when the set $X$ has empty interior and $\p \neq\{0\}$.  We describe an algebraic reduction of \eqref{eq:gmp_relax} using the equality constraints $p_1,\dots,p_s=0$. For ease of notation, we fix $I:=\Ip = (p_1,\dots,p_s)$. 

Note that in Assumption \ref{hyp:ball}, we only have only one inequality constraint, which is redundant. The reduction described below is valid outside the duality attainment context of this paper and so we maintain generality and keep an arbitrary number of inequality constraints given by $g_1\dots,g_\ell$, as in the notation of \eqref{eq:dual_relax}.

The truncation of the polynomial ring $\R[x]_t$ is a finite dimensional real vector space. Since we are interested in $X= V_\R(I)$, we work over the quotient ring $\R[x]/I$, which is the set of elements of the form $f + I = \{f + k \mid k \in I\}$. To define a truncation of $\R[x]/I$, we first recall the truncation of the ideal $I$ as,
\[
I_{t} := I \cap \R[x]_t.
\]
Note that, one can consider the set $\langle \p \rangle_t = \{\sum_{l=1}^s k_l p_l  \mid \deg(k_l p_l) \leq t\}$ as another way to truncate $I$, which is equal to $I_t$ when the $p_l$'s form a Gröbner basis, see Lemma \ref{le:truncated_ideal}. The truncation of the quotient ring is then realized as the quotient vector space $\R[x]_t/I_t$. If two polynomials $f,g\in \R[x]_t$ are in the same equivalence class of $\R[x]_t/I_t$, we write $f\equiv g$, and denote the equivalence class by $[f]$. Note that equivalence classes of $\R[x]_t/I_t$ may be different from those of $\R[x]/I$.

Recall that, the set of monomials $\{x^{\alpha}\}_{\alpha\in \N^n_{t}}$ is a basis of $\R[x]_{t}$ and indexes the pseudo moment matrix $M_t(\mom)$, with $\mom\in\R^{\N^n_{2t}}$. To work over the truncated quotient $\R[x]_t/I_t$, we need to describe a basis for it. We have from \cite[Chapter 5.3]{cox_ideals_2015} that the following set gives a basis for the untruncated quotient ring $\R[x]/I$,
\[
\mathcal{B} = \{x^\alpha \mid x^\alpha \notin \LT(I)\},
\]
see also \cite[Chapter 2.3]{laurent_survey_2009}. Note that, this basis depends on the choice of monomial ordering, and recall that we work with the graded lexicographical ordering on $\R[x]$. We truncate
\[
\mathcal{B}_t = \mathcal{B} \cap \R[x]_t.
\]
\begin{lemma}
\label{lemma:bt}
We have that $\mathcal{B}_t$ is a basis of $\R[x]_t/I_t$, i.e., $\Span_{\R}\big(\mathcal{B}_t\big) = \R[x]_t/I_t$ and $\mathcal{B}_t$ is linearly independent.
\end{lemma}
\begin{proof}
First, we show that $\mathcal{B}_t$ spans $\R[x]_t / I_t$. Let $f \in \mathbb{R}[x]_t$, with equivalence class $[f] \in \R[x]_t / I_t$. Let $G = \{k_1,\dots, k_g\}$ be a Gröbner basis of $I$, with respect to the graded lexicographical ordering. By multivariate division with respect to $G$, we can write
\[
f = \sum_{l=1}^g \phi_l k_l + \psi,
\]
where $\psi$ is the remainder, whose monomials lie in $\mathcal{B}$.
Since the monomial order is degree-compatible, $\deg(\psi) \le t$, hence $\psi \in \Span_{\R}(\mathcal{B}_t)$.
Moreover, $f - \psi \in I$ and $\deg(f-\psi) \le t$, so $f - \psi \in I_t$. Thus $f \equiv \psi$ and $[f] \in \Span_{\R}(\mathcal{B}_t)$.

Next, we show linear independence. Suppose that
\[
f = \sum_{x^\alpha \in \mathcal{B}_t} c_\alpha x^\alpha \in I_t.
\]
Then $f \in I$. If $f \neq 0$, its leading term $\LT(f)$ must lie in $\LT(I)$. However, all monomials of $f$ belong to $\mathcal{B}$, hence none lies in $\LT(I)$, a contradiction. Therefore $f = 0$, and all coefficients $c_\alpha$ vanish. This proves linear independence. 
Hence $\mathcal{B}_t$ is a basis of $\mathbb{R}[x]_t / I_t$.
\end{proof}
We give an example of such a basis below, in Remark \ref{rk:spheres_basis}.

When working with pseudo moment matrices, we used $|\N^n_t|$ and $|\N^n_{2t}|$. Similarly, we will use $\mathcal{B}_t = \{q_1,\dots,q_{r_{t}}\}$ and $\mathcal{B}_{2t} = \{q_1,\dots,q_{r_{2t}}\}$. While $\mathcal{B}_t$ is unique up to the choice of monomial ordering, $r_{t} = |\mathcal{B}_t|$ is unique for any $t$.  Note that, we have that $r_t\leq |\N^n_t|$. 

For $k$ in $\R[x]_{2t}$, the above allows us to write $k \equiv \sum_{i=1}^{r_{2t}} c_i^{(k)} q_i$, for some real coefficients $c_i^{(k)}$. Set $c^{(k)} = (c_i^{(k)})$, as the vector of coefficients of the residual of $k$. We fix the notation 
\begin{equation}\label{eq:c_ialpha}
x^\alpha \equiv \sum_{i=1}^{r_{2t}} c_{i,\alpha} q_i.
\end{equation} 

The basis construction above allows us to define the reduced version of the objects discussed in the introduction, which we will need to formulate the reduced semidefinite programs. We follow a similar construction as the one in \cite{laurent2007varieties}, which deals with an untruncated finite dimensional quotient ring (when $V_{\R}(I)$ is zero dimensional). We define the \emph{reduced Riesz functional} of $\hat{\mom}\in \R^{\mathcal{B}_{2t}}$ as 
\[
\hat{L}_{\hat{\mom}}(k) = \sum_{i=1}^{r_{2t}} c_i^{(k)} \hat{\mom}_i,\;\; k\in \R[x]_{2t}/I_{2t}.
\]
This allows us to define the \emph{reduced pseudo moment matrix} as the $r_t\times r_t$ matrix with the following entries,
\begin{equation}\label{eq:reduced_mom_matrix}
\hat{M}_t(\hat{\mom})_{ij} = \hat{L}_{\hat{\mom}}(q_i q_j) = \sum_{e=1}^{r_{2t}} c_e^{(q_i q_j)} \hat{\mom}_e.
\end{equation}
Furthermore, given a polynomial $g \equiv \sum_l c^{(g)}_l q_l \in \R[x]_{2t}/I_{2t}$, set $t_g = \lceil\frac{\deg(g)}{2}\rceil$. We define the \emph{reduced shifted sequence}
\begin{equation}\label{eq:shifted_seq_reduced}
(g \hat{\mom})_i = \hat{L}_{\hat{\mom}}(g q_i) = \sum_l c^{(g)}_l \hat{L}_{\hat{\mom}}(q_i q_l) ,\;\; \text{with}\; i=1,\dots,r_{2(t-t_g)}.
\end{equation}
This in turn allows us to define the \emph{reduced localizing matrix} associated to $g$ and $\hat{z}$ as the reduced pseudo moment matrix of the sequence $g \hat{\mom}$, and is written as $\hat{M}_{(t-t_g)}(g\hat{\mom})$. Thus, 
\[
\hat{M}_{(t-t_g)}(g\hat{\mom})_{ij} = \big(\hat{L}_{\hat{\mom}}(g q_i q_j)\big)= \sum_{l} c_l^{(g)} \hat{L}_{\hat{\mom}}(q_i q_j q_l) ,\;\; \text{with}\; i,j=1,\dots,r_{2(t-t_g)}.
\]

Let $\Sigma[x]'_{2t}$ be the set of sums of squares of polynomials supported on $\mathcal{B}_t$. 
Similarly with $g$ as above, let $\Sigma[x]'_{2(t-t_{g})}$ be the set of sums of squares of polynomials supported on $\mathcal{B}_{t-t_{g}}$.

With these notions, we write the reduced versions of \eqref{eq:gmp_relax} and 
\eqref{eq:dual_relax} as
\begin{equation}\label{eq:reduced_primal}
\begin{aligned}
\primal_t^{\red} = &\inf_{\hat{\mom}}  \; \; \hat{L}_{\hat{\mom}}(f) \\
& \st \;  \hat{L}_{\hat{\mom}}(h_i) = b_i, \; i=1,\dots,m\\
& \qquad \hat{M}_t(\hat{\mom})\succeq 0, \;\hat{M}_{t-t_{g_j}}(g_j\hat{\mom}) \succeq 0,\\
& \qquad j = 1,\dots,\ell,\;\; \hat{\mom}\in \R^{\mathcal{B}_{2t}},
\end{aligned}
\end{equation}
and
\begin{equation}\label{eq:reduced_dual}
\begin{aligned}
\dual_t^{\red} = &\sup_{\lambda,\{\sigma_j\}}  \; \; \sum_{i=1}^m b_i \lambda_i \\
& \st \;  f - \sum_{i=1}^m  h_i\lambda_i \equiv \sigma_0 + \sum_{j=1}^\ell \sigma_j g_j \\
& \qquad  \sigma_0 \in \Sigma[x]'_{2t}, \;\; \sigma_j \in \Sigma[x]'_{2(t-t_{g_j})},\\
& \qquad j = 1,\dots,\ell,\;\;\lambda \in \R^m,
\end{aligned}
\end{equation}
respectively. 

We now draw a connection between the unreduced and reduced moment relaxations. Given a sequence $\hat{\mom}\in \R^{\mathcal{B}_{2t}}$, we extend it to a sequence $\mom \in \R^{\N^n_{2t}}$ by using \eqref{eq:c_ialpha}. For $\alpha\in \N^n_{2t}$, we set
\begin{equation}\label{eq:mom_extension}
\mom_{\alpha} := \sum_{i=1}^{r_{2t}} c_{i,\alpha} \,\hat{\mom}_i.
\end{equation}
Let $U_{2t}$ denote the $r_{2t} \times |\N^n_{2t}|$ matrix whose $(i,\alpha)$ entry is equal to $c_{i,\alpha}$. Hence, we have $\mom = U_{2t}^T \hat{\mom} \in \R^{\N^n_{2t}}$. For a polynomial $k= \sum_{\alpha\in \N_{2t}^n}k_{\alpha}x^{\alpha}$ in $\R[x]_{2t}$, we can write
\begin{equation}\label{eq:U_polynomial}
c^{(k)} = U_{2t} k,
\end{equation}
where the symbol $k$ is also used to refer to the vector of coefficients $(k_\alpha)_{\alpha\in \N_{2t}^n}$. This uses the uniqueness property of $\mathcal{B}_{2t}$, as follows,
\begin{equation}
k = \sum_\alpha k_\alpha x^\alpha \equiv \sum_\alpha k_\alpha \sum_{i=1}^{r_{2t}} c_{i,\alpha} q_i =\sum_{i=1}^{r_{2t}} \Big(  \sum_\alpha k_\alpha c_{i,\alpha} \Big) q_i \equiv \sum_{i=1}^{r_{2t}} c_i^{(k)} q_i.
\end{equation}
Note that $U_{2t}$ has full row rank. We similarly let $U_t$ be the $r_{t} \times |\N^n_{t}|$ submatrix of $U_{2t}$ given by the first $r_t$ rows and $|\N^n_{t}|$ columns. 

The following lemma resolves a potential conflict between the shifted sequence of an extension and the extension of a reduced shift sequence.

\begin{lemma}\label{le:reduced_shifted_seqs}
Let $\hat{\mom}$ be a sequence in $\R^{\mathcal{B}_{2t}}$, and $\mom$ be a sequence in $\R^{\N^n_{2t}}$, with $\mom = U_{2t}^T \hat{\mom}$, as in \eqref{eq:mom_extension}. Further let $g$ be a polynomial in $\R[x]_{2t}$, with $t_g = \lceil\frac{\deg(g)}{2}\rceil$.
\begin{enumerate}[(i)]
\item $\hat{L}_{\hat{\mom}}(g) = L_{\mom}(g)$.
\item The shifted sequence $g\mom$, as in \eqref{eq:shifted_seq}, is equal to the extension of the reduced shifted sequence $g\hat{\mom}$, as in  \eqref{eq:shifted_seq_reduced}, i.e.,
\[
(g\mom)_\alpha = \big( U^T_{2(t-t_g)} \, g\hat{\mom} \big)_\alpha,\;\;\text{for all}\;\; \alpha \in \N^n_{2(t-t_g)}.
\]
\end{enumerate}
\end{lemma}
\begin{proof}
\begin{enumerate}[(i)]
    \item Using \eqref{eq:U_polynomial}, we write,
\[
L_{\mom}(g) = g^T z = g^T (U^T_{2t} \hat{\mom}) = (c^{(g)})^T \hat{\mom} = \hat{L}_{\hat{\mom}}(g).
\]
\item For $\alpha \in \N^n_{2(t-t_g)}$, we have $(g\mom)_\alpha = L_\mom(g x^\alpha)$. On the other hand, we write
\[
\big( U^T_{2(t-t_g)} \, g\hat{\mom} \big)_\alpha =  \sum_{i=1}^{r_{2(t-t_g)}} c_{i,\alpha} \hat{L}_{\hat{\mom}}(q_i g) = \hat{L}_{\hat{\mom}}\Big( \big(\sum_{i=1}^{r_{2(t-t_g)}} c_{i,\alpha} q_i \big) g \Big) = \hat{L}_{\hat{\mom}}(g x^\alpha),
\]
and the claim follows by the first item. \qedhere
\end{enumerate}
\end{proof}

The next lemma shows a relation between the reduced and unreduced pseudo moment and localizing matrices of $\hat{\mom}$ and $\mom$, respectively.

\begin{lemma}\label{le:moment_matrices}
Consider the assumptions of Lemma \ref{le:reduced_shifted_seqs}.
\begin{enumerate}[(i)]
    \item We have that $M_t(\mom) = U_t^T\hat{M}_t(\hat{\mom})U_t$, and \[\hat{M}_t(\hat{\mom})\succeq 0\;\; \text{if and only if}\;\; M_t(\mom)\succeq 0.\]
    \item Similarly, we have that $M_{t-t_g}(g\mom) = U_{t-t_g}^T\hat{M}_{t-t_g}(g\hat{\mom})U_{t-t_g}$, and \[\hat{M}_{t-t_g}(g\hat{\mom})\succeq 0\;\; \text{if and only if}\;\; M_{t-t_g}(g\mom)\succeq 0.\]
\end{enumerate}
\end{lemma}
\begin{proof}
\begin{enumerate}[(i)]
\item The following computation is adapted from \cite[Lemma 13]{laurent2007varieties}. Fix $\alpha,\beta \in \N^n_t$. We have
\[
\begin{aligned}
\Big(U_t^T\hat{M}_t(\hat{\mom})U_t\Big)_{(\alpha,\beta)} =& \sum_{i,j=1}^{r_t} U_{t\,(i,\alpha)} \hat{M}_t(\hat{\mom})_{ij} U_{t\, (j, \beta)} \\
=& \sum_{i,j=1}^{r_t} c_{i,\alpha} \Big(\sum_{e=1}^{r_{2t}} c_e^{(q_i q_j)} \hat{\mom}_e\Big) c_{j, \beta}\\
=& \sum_{e=1}^{r_{2t}} \Big(\sum_{i,j=1}^{r_t} c_{i,\alpha} c_{j, \beta} c_e^{(q_i q_j)}\Big) \hat{\mom}_e.
\end{aligned}
\]
From \eqref{eq:c_ialpha}, we directly have that
\[
\begin{aligned}
x^{\alpha}x^{\beta} \equiv& \Big(\sum_{i=1}^{r_t} c_{i,\alpha}\, q_i\Big)\Big(\sum_{j=1}^{r_t} c_{j,\beta}\, q_j\Big)\\
=& \sum_{i,j=1}^{r_t} c_{i,\alpha} \,c_{j,\beta}\, q_i\, q_j\\
\equiv& \sum_{e=1}^{r_{2t}} \Big(\sum_{i,j=1}^{r_t} c_{i,\alpha} \, c_{j,\beta}\, c^{(q_i q_j)}_e\Big) q_e,
\end{aligned}
\]
which is equivalent to
\[
x^{\alpha+\beta} \equiv \sum_{i=1}^{r_{2t}} c_{i,\alpha+\beta} q_i.
\]
Together, on the moment side, the above gives
\[
M_t(\mom)_{(\alpha,\beta)} = \mom_{\alpha+\beta} =  \sum_{i=1}^{r_{2t}} c_{i,\alpha+\beta} \hat{\mom}_i = \sum_{e=1}^{r_{2t}} \Big(\sum_{i,j=1}^{r_t} c_{i,\alpha} \, c_{j,\beta}\, c^{(q_i q_j)}_e\Big) \hat{\mom}_e,
\]
where we use the uniqueness property of the basis $\mathcal{B}_{2t}$. Thus, we have shown that $M_t(\mom) = U_t^T\hat{M}_t(\hat{\mom})U_t$, and the claim follows using the fact that $U_t$ has full row rank.

\item The proof follows by a similar computation. With $g=\sum_{\gamma\in\N^n_{2t_g}} g_{\gamma}x^{\gamma} \equiv \sum_{l=1}^{r_{2t_g}} c^{(g)}_l q_l $, fix $\alpha,\beta \in \N^n_{t-t_g}$. We have
\[
\begin{aligned}
\Big(U_{t-t_g}^T\hat{M}_{t-t_g}(\hat{\mom})U_{t-t_g}\Big)_{(\alpha,\beta)} =& \sum_{i,j=1}^{r_{t-t_g}} U_{t-t_g\,(i,\alpha)} \hat{M}_{t-t_g}(\hat{\mom})_{ij} U_{t-t_g\, (j, \beta)} \\
=& \sum_{i,j=1}^{r_{t-t_g}} c_{i,\alpha} \Big(\sum_{l=1}^{r_{2t_g}} c_l^{(g)} \hat{L}_{\hat{\mom}}(q_i q_j q_l) \Big) c_{j, \beta}\\
=& \sum_{i,j=1}^{r_{t-t_g}} c_{i,\alpha} \,c_{j, \beta} \Big(\sum_{l=1}^{r_{2t_g}} c_l^{(g)} \big( \sum_{e=1}^{r_{2t}} c^{(q_i q_j q_l)}_e \hat{\mom}_e \big) \Big)\\
=& \sum_{i,j=1}^{r_{t-t_g}} \sum_{l=1}^{r_{2t_g}} \sum_{e=1}^{r_{2t}}  \sum_{\gamma\in\N^n_{2t_g}} c_{i,\alpha}\, c_{j, \beta} \,c_{l,\gamma} \,g_{\gamma} \,c^{(q_i q_j q_l)}_e \hat{\mom}_e,
\end{aligned}
\]
where the last equality holds by \eqref{eq:U_polynomial}. For $\gamma\in\N^n_{2t_g}$, we have directly from \eqref{eq:c_ialpha} that
\[
\begin{aligned}
x^{\alpha}x^{\beta}x^{\gamma} \equiv& \Big(\sum_{i=1}^{r_{t-t_g}} c_{i,\alpha}\, q_i\Big)\Big(\sum_{j=1}^{r_{t-t_g}} c_{j,\beta}\, q_j\Big)\Big(\sum_{l=1}^{r_{2t_g}} c_{l,\gamma}\, q_l\Big)\\
=& \sum_{i,j=1}^{r_{t-t_g}}\sum_{l=1}^{r_{2t_g}} c_{i,\alpha} \,c_{j,\beta}\, c_{l,\gamma} q_i\, q_j\, q_l\\
\equiv& \sum_{e=1}^{r_{2(t-t_g)}} \Big(\sum_{i,j=1}^{r_{t-t_g}} \sum_{l=1}^{r_{2t_g}} c_{i,\alpha} \, c_{j,\beta}\,c_{l,\gamma} \,c^{(q_i q_j q_l)}_e\Big) q_e,
\end{aligned}
\]
which is equivalent to
\[
x^{\alpha+\beta+\gamma} \equiv \sum_{d=1}^{r_{2t}} c_{d,\alpha+\beta+\gamma}\, q_d.
\]
Together, on the moment side, the above gives
\[
\begin{aligned}
M_{t-t_g}(\mom)_{(\alpha,\beta)} =& \sum_{\gamma\in\N^n_{2t_g}} g_\gamma \mom_{\alpha+\beta+\gamma} \\
=& \sum_{\gamma\in\N^n_{2t_g}} g_\gamma \Big( \sum_{d=1}^{r_{2(t-t_g)}} c_{d,\alpha+\beta+\gamma} \hat{\mom}_d \Big)\\ 
=& \sum_{\gamma\in\N^n_{2t_g}} g_\gamma \Big(\sum_{e=1}^{r_{2t}} \Big(\sum_{i,j=1}^{r_{t-t_g}} \sum_{l=1}^{r_{2t_g}} c_{i,\alpha} \, c_{j,\beta}\,c_{l,\gamma} \,c^{(q_i q_j q_l)}_e\Big) \hat{\mom}_e \Big).
\end{aligned}
\]
Thus, we have similarly shown that $M_{t-t_g}(\mom) = U_{t-t_g}^T\hat{M}_{t-t_g}(\hat{\mom})U_{t-t_g}$ and the claim follows since $U_{t-t_g}$ has full row rank. \qedhere
\end{enumerate}
\end{proof}

The following statement shows the relation between the reduced and unreduced semidefinite programs.

\begin{proposition}\label{pr:reduced_programs}
Suppose the given generating set $p_1\dots,p_s$ of $I$ forms a Gröbner basis with respect to the graded lexicographic order. 
With $\primal_t$ as in \eqref{eq:gmp_relax} and $\primal_t^{\red}$ as in  \eqref{eq:reduced_primal}, we have that $\primal_t \leq \primal_t^{\red}$. 
With $\dual_t$ as in \eqref{eq:dual_relax} and $\dual_t^{\red}$ as in  \eqref{eq:reduced_dual}, we have that $\dual_t^{\red} = \dual_t$.
\end{proposition}
\begin{proof}
We show $\primal_t \leq \primal_t^{\red}$ by fixing a feasible solution $\hat{\mom}\in\R^{\mathcal{B}_{2t}}$ of \eqref{eq:reduced_primal} and showing it can be extended to a feasible solution of \eqref{eq:gmp_relax}, with the same objective value. Let $\mom\in \R^{\N_{2t}^n}$ be the extension of $\hat{\mom}$, as in \eqref{eq:mom_extension}. By Lemma \ref{le:moment_matrices}, the positive semidefinite constraints in \eqref{eq:gmp_relax} are satisfied. Furthermore, we have that $\hat{L}_{\hat{\mom}}(k) = L_{\mom}(k)$ for any polynomial $k\in \R[x]_{2t}$, by Lemma \ref{le:reduced_shifted_seqs}. This implies that the linear moment equality constraints are satisfied, and that the objective values of $\mom$ and $\hat{\mom}$, in their respective programs, are the same.

It remains to show that the constraints $M_{t-t_{p_j}}(p_j \mom) = 0$ with $j=1,\dots,s$ in \eqref{eq:gmp_relax} are satisfied. Fix $k\in I_{2t}$, and set $t_k = \lceil\frac{\deg k}{2}\rceil$. Since the residue of $k$ is zero in $\R[x]_{2t}/I_{2t}$, we have that $c^{(k)}\in \R^{r_{2t}}$ is the zero vector. For $\alpha\in \N^n_{2(t-t_k)}$, we have
\[
(k\mom)_\alpha = \sum_{\gamma\in \N^n_{t_k}} k_\gamma \mom_{\alpha+\gamma},
\]
using \eqref{eq:shifted_seq}. Now using the definition of the extension \eqref{eq:mom_extension} and the definition of the matrix $U_{2t}$, the above is equal to, 
\[
\sum_{\gamma\in \N^n_{t_k}} \sum_{i=1}^{r_{2t}} k_\gamma c_{i,\alpha + \gamma} \hat{\mom}_i =  \sum_{i=1}^{r_{2t}} \Big( \sum_{\gamma\in \N^n_{t_k}} (U_{2t})_{(i, \alpha+\gamma)} k_{\gamma} \Big) \hat{\mom}_i.
\]
Finally, using \eqref{eq:U_polynomial}, we have 
\[
\sum_{\gamma\in \N^n_{t_k}} (U_{2t})_{(i, \alpha+\gamma)} k_{\gamma} = c^{(x^\alpha k)}_i = 0,
\]
where the last equality holds since $x^\alpha k\in I_{2t}$, for all $\alpha\in \N^n_{2(t-t_k)}$. Therefore, each entry of the shifted sequence $k\mom$ is zero, and so $M_{t-t_k}(k\mom) =0$ for any $k\in I_{2t}$. Since the polynomials $p_1,\dots,p_s$ are in $I_{2t}$ by definition, the constraints in \eqref{eq:gmp_relax} hold. Hence, $\mom$ is feasible for \eqref{eq:gmp_relax} and $\primal_t \leq \primal_t^{\red}$.

For the dual programs, first note that the objective functions yielding the values $\dual_t^{\red}$ and $\dual_t$ are the same. Furthermore, we have that a feasible solution of the reduced problem gives a feasible solution to the unreduced problem, and vice versa. This is by the equality stated in Lemma \ref{le:truncated_ideal}, which is assured by the Gröbner basis assumption on the given generating set of $I$,
\[
I_{2t} = \left\{\sum_{l=1}^s k_l p_l  \mid \deg(k_l p_l)\leq 2t\right\}.
\]
Hence, we have that $\dual_t^{\red} = \dual_t$.
\end{proof}

\subsection{Proof of the case of empty interior}\label{ap:pfofthm}

We use the notions developed in the previous appendix, to complete the proof in Subsection \ref{ssec:GMP_relax_density_proof}. We repeat the statement.

\begin{theorem*}[Repetition of Theorem \ref{th:attainSOS}(i)]
Let $X$ be a compact nonempty set which satisfies Assumption \ref{hyp:ball} (i). 
Suppose the ideal $\Ip=(p_1,\dots,p_s)$ is real radical, and the polynomials $p_i$ form a Gröbner basis with respect to the graded lexicographic order. Assume further, that $h_1\equiv1$, $b_1 > 0$ and that $b = (b_1,\dots,b_m)$ is in the relative interior of $K$. Then the semidefinite program \eqref{eq:dual_relax} attains its optimal value for all $t \geq t_{\min}$.
\end{theorem*}
\begin{proof}
From the weak duality of semidefinite programs, we have that $\dual_t\leq \primal_t$. From Proposition \ref{pr:reduced_programs}, we have that $\primal_t \leq \primal_t^{\red}$ and $\dual_t^{\red} = \dual_t$. 
We now show, using the strong duality of semidefinite programs, that $\primal_t^{\red}=\dual_t^{\red}$ and that the optimal value $\dual_t^{\red}$ is attained. 

By Corollary \ref{cor:feasible_measure}, we have the existence of a feasible measure $\nu$ on $X$. Let $\mom^{\nu}$ be the associated moment sequence, i.e., $z^\nu_\alpha = \int x^\alpha \di \nu$, with $\alpha\in \N^n_{2t}$. With $k\in \Ip$ and using \eqref{eq:c_ialpha}, we can write
\[
z^\nu_\alpha = \int x^\alpha \di \nu = \int \big(\sum_{i=1}^{r_{2t}} c_{i,\alpha} q_i + k \big) \di \nu = \sum_{i=1}^{r_{2t}} c_{i,\alpha} \underbrace{\int q_i  \di \nu}_{=: \hat{\mom}_{i}^{\nu}} ,
\]
since $k$ vanishes on $X$, where $\nu$ is supported. Hence, we have that $\mom^{\nu} = U_{2t}^T \hat{\mom}^{\nu}$, and, by the argument in the proof of Proposition \ref{pr:reduced_programs}, $\hat{\mom}^{\nu}$ is feasible for \eqref{eq:reduced_primal}. 

It remains to show that $\hat{\mom}^{\nu}$ is strictly feasible, that is, $\hat{M}_t(\hat{\mom}^{\nu}) \succ 0$ and $\hat{M}_{t-t_{g_1}}(\hat{\mom}^{\nu}) \succ 0$. To show $\hat{M}_t(\hat{\mom}^{\nu}) \succ 0$, it is enough to show $(c^{(\psi)})^T \hat{M}_t(\hat{\mom}^{\nu}) c^{(\psi)}=0$ implies $c^{(\psi)}=0$, for $\psi = \sum_i c^{(\psi)}_i q_i \in \R[x]_t/I_{t}$. We write
\[
0 = (c^{(\psi)})^T \hat{M}_t(\hat{\mom}^{\nu}) c^{(\psi)} = \int \psi^2 + k \di \nu = 0, \;\; \text{for some}\;\; k\in \Ip.
\]
Since $k$ vanishes on $X$, we have that $\psi = 0$ on $X$, and thus $c^{(\psi)} = 0$. By Assumption \ref{hyp:ball}, we have that $g_1>0$ on $V_{\R}(\Ip)$, and so $\hat{M}_{t-t_{g_1}}(g_1 \hat{\mom}^{\nu}) \succ 0$ follows by a similar computation.

Thus, the Slater condition is satisfied by the primal \eqref{eq:reduced_primal}, so the strong duality of semidefinite programs implies that the dual problem attains its optimal value and $\primal_t^{\red}=
\dual_t^{\red}$. So, by the above, we have that $\primal_t=
\dual_t$.

It remains to show the attainment of the optimal value $\dual_t$. For this we need to show that the optimal solution to the program \eqref{eq:reduced_dual} (the one which attains the value $\dual_t^{\red}$) also gives a solution to the unreduced problem. This was already shown in the proof of Proposition \ref{pr:reduced_programs}, yielding the desired result. 
\end{proof}

\begin{remark}
The proof above strengthens the result of Proposition \ref{pr:reduced_programs}, in that it shows that $\primal_t = \primal_t^{\red}$. This equality can be shown directly (without referring to the dual program) by taking a feasible $\mom \in \R^{\N^n_{2t}}$, and showing that there exists $\hat{\mom}\in \R^{\mathcal{B}_{2t}}$ such that $\mom = U_{2t}^T \hat{\mom}$. This can be shown using the fact that $\mom$ is feasible and hence satisfies the equality constraints in \eqref{eq:gmp_relax}.
\end{remark}

\subsection{Algebraic properties of products of spheres}\label{ap:spheres_alg}

In Subsection \ref{ssec:sphs}, we apply Theorem \ref{th:attainSOS} to the case of the product of spheres. To do this, we need the following two statements, which fulfill the assumptions needed. The following lemma proves the real radicality assumption of Theorem \ref{th:attainSOS}. We work with the notation introduced in Subsection \ref{ssec:sphs}. 

\begin{lemma}\label{le:realrad}
The ideal 
\[I(\bm{\varphi}_s) = (\varphi_1,\dots,\varphi_s) =\big(1-\|x_1\|^2, \dots, 1-\|x_s\|^2\big)\] in $\R[x_1,\dots,x_s]$ is real radical.
\end{lemma}
\begin{proof}
For ease of notation, we set $I(\bm{\varphi}_s) = J_s$. We proceed by induction on $s$. For the case $s=1$, we have that $J_1=(1-\|x_1\|^2)$ is a prime ideal, since the polynomial $1-\|x_1\|^2$ is irreducible. Hence, it is radical, i.e., $\mathcal{I}\big(V_{\C}(J_1)\big) = J_1$. Further, we have that $V_{\C}(J_1)$ contains a real smooth point, and so the real vanishing set of $J_1$, $V_{\R}(J_1)$, is Zariski dense in the complex vanishing set, see \cite[Theorem 5.1]{sottile2019realalggeo}. This immediately implies that $\mathcal{I}\big(V_{\C}(J_1)\big) = \mathcal{I}\big(V_{\R}(J_1)\big)$. By the Realnullstellensatz and the radicalness of $J_1$, the result follows. For details, see \cite[Section 3]{schweighofer2022realalggeo}.

For the inductive step, consider $k\in \R[x_1,\dots,x_s]$ such that $k$ vanishes on the product of spheres $X$. We write 
\[
k(x_1,\dots,x_s) = \sum_{\alpha\in \N^{n_s}} \phi_{\alpha}(x_1,\dots,x_{s-1})x_s^{\alpha},
\]
for $\phi_\alpha\in \R[x_1,\dots,x_{s-1}]$. Fix $(\hat{x}_1,\dots,\hat{x}_{s-1}) \in \mathbb{S}^{n_1-1}\times \dots \times \mathbb{S}^{n_{s-1}-1}$, so that
\[
\hat{k}(x_s) = k(\hat{x}_1,\dots,\hat{x}_{s-1},x_s) = \sum_{\alpha\in \N^{n_s}} \phi_{\alpha}(\hat{x}_1,\dots,\hat{x}_{s-1})x_s^{\alpha},
\]
is an element of $\R[x_s]$, which vanishes on $\mathbb{S}^{n_s-1}$. By the argument above, we have that the ideal $I = (1-\|x_s\|^2)\subset \R[x_s]$ is real radical. Thus $\hat{k}\in I$, and $\hat{k}(x_s) = (1-\|x_s\|^2)\theta_{\hat{x}_1\dots\hat{x}_{s-1}}(x_s)$, for some $\theta_{\hat{x}_1\dots\hat{x}_{s-1}}\in \R[x_s]$. Setting $\theta_{\hat{x}_1\dots\hat{x}_{s-1}}(x_s) = \sum_{\beta\in \N^{n_s}} \theta_{\hat{x}_1\dots\hat{x}_{s-1},\beta}\,x_s^{\beta}$, we can write
\[
\begin{aligned}
\hat{k}(x_s) &= (1-\|x_s\|^2)\theta_{\hat{x}_1\dots\hat{x}_{s-1}}(x_s) \\
&= (1-\|x_s\|^2) \sum_{\beta\in \N^{n_s}} \theta_{\hat{x}_1\dots\hat{x}_{s-1},\beta}\,x_s^{\beta} = \sum_{\alpha\in \N^{n_s}} \phi_{\alpha}(\hat{x}_1,\dots,\hat{x}_{s-1})\,x_s^{\alpha}.
\end{aligned}
\]
Expanding the left hand side, gives a sum over $\beta\in \N^{n_s}$, where the coefficients of the $x^{\beta}$'s are simple linear combinations of the $\theta_{\hat{x}_1\dots\hat{x}_{s-1},\beta}$'s. So by comparing coefficients, we obtain a system of equations, which, when solved, shows that the $\theta_{\hat{x}_1\dots\hat{x}_{s-1},\beta}$'s, have polynomial dependence on  $\hat{x}_1,\dots,\hat{x}_{s-1}$. Thus, we can write $\theta_{\hat{x}_1\dots\hat{x}_{s-1}}(x_s) = q_s(\hat{x}_1,\dots,\hat{x}_{s-1},x_s)$ which gives,
\[
k(\hat{x}_1,\dots,\hat{x}_{s-1}, x_s) = (1-\|x_s\|^2) q_s(\hat{x}_1,\dots,\hat{x}_{s-1},x_s),
\]
for $q_s\in \R[x_1,\dots,x_s]$. Now for a fixed $\hat{x}_s\in \R^{n_s}$, we have that the polynomial,
\begin{equation}\label{eq:radical_polynomial}
x_1,\dots, x_{s-1} \longmapsto k(x_1,\dots,x_{s-1}, \hat{x}_s) - (1-\|\hat{x}_s\|^2) q_s(x_1,\dots,x_{s-1},\hat{x}_s)
\end{equation}
vanishes on the set $\mathbb{S}^{n_1-1}\times \dots \times \mathbb{S}^{n_{s-1}}$. By the inductive hypothesis, the polynomial described in \eqref{eq:radical_polynomial} is in the ideal $J_{s-1}=\big(1-\|x_1\|^2, \dots, 1-\|x_{s-1}\|^2\big)$. So we have that 
\[
\begin{aligned}
k(x_1,\dots,x_{s-1}, \hat{x}_s) &- (1-\|\hat{x}_s\|^2) q_s(x_1,\dots,x_{s-1},\hat{x}_s) \\
&= \sum_{i=1}^{s-1} (1-\|x_i\|^2) q_{\hat{x}_s, i}(x_1,\dots,x_{s-1}),
\end{aligned}
\]
with $q_{\hat{x}_s, i} \in \R[x_1,\dots,x_{s-1}]$, for $i=1,\dots,s-1$. By a similar argument as above, we have that the $q_{\hat{x}_s, i}$'s have polynomial dependence on $\hat{x}_s$, i.e., $q_{\hat{x}_s, i} (x_1,\dots,x_{s-1}) = q_i(x_1,\dots,x_{s-1},\hat{x}_s)$, for any $\hat{x}_s\in \mathbb{S}^{n_s-1}$. This shows that $k\in J_s$ and proves the claim.
\end{proof}

The next lemma proves the Gröbner assumption of Theorem \ref{th:attainSOS} for the generating set $\{\varphi_1,\dots,\varphi_s\}$. For this we introduce the notion of coprime monomials. We say two monomials $x^\alpha$ and $x^\beta$, with $\alpha,\beta \in \mathcal{N}$, are \emph{coprime}, if their greatest common divisor is one.

\begin{lemma}\label{le:grobner_basis}
The set of polynomials $\Gb = \{\varphi_1,\dots,\varphi_s\}\subset \R[x]$, with $\varphi_i = 1 - \|x_i\|^2$, form a Gröbner basis of the ideal $J_s = (\varphi_1,\dots,\varphi_s)$.
\end{lemma}
\begin{proof}
We use the fact that, if the leading terms of the polynomials $\varphi_1,\dots,\varphi_s$ are pairwise coprime, then $\Gb$ is a Gröbner basis for $J_s$. See \cite[Corollary 2.5.10]{kreuzer_compalg_2008}. Since $\LT(\varphi_i) = x_{i1}^2$, for each $i=1,\dots,s$, this property is satisfied.
\end{proof}

From Subsection \ref{ap:algred}, recall the notion of a basis $\mathcal{B}_t$ of the truncated quotient $\R[x]_t/(J_s)_t$. To close, we give an example of such a basis for the case of the product of spheres, where the elements of $\mathcal{B}_t$ are monomials. This is adapted from the proof of \cite[Lemma 1]{henrion2012innerapprox}.

\begin{remark}\label{rk:spheres_basis}
Recall that $\mathcal{N} = \N^{n_1+\dots + n_s}$ and $\mathcal{N}_t = \N_t^{n_1+\dots + n_s}$. 
By Lemma \ref{le:grobner_basis}, the set of polynomials $\Gb = \{\varphi_1,\dots,\varphi_s\}$, with $\varphi_i = 1 - \|x_i\|^2$, form a Gröbner basis of the ideal $J_s = (\varphi_1,\dots,\varphi_s)$. 
In addition, $\LT(\varphi_i) = x_{i1}^2$, for each $i=1,\dots,s$. 
By the property of Gröbner bases, one has $\LT(J_s) = \LT(\Gb)$, so $\mathcal{B} = \big\{x^{\alpha} \; | \; \alpha=(\alpha_1, \dots, \alpha_s) \in \mathcal{N}, \; \alpha_{i 1} \leq 1, \; \text{for each}\; i=1,\dots,s \big\}$. 
Then one has 
$\mathcal{B}_t = \big\{x^{\alpha} \; | \; \alpha=(\alpha_1, \dots, \alpha_s) \in \mathcal{N}_t, \; \alpha_{i 1} \leq 1, \; \text{for each}\; i=1,\dots,s \big\}$. 
\end{remark}


\end{document}